\documentclass[a4paper, reqno]{amsart}
\usepackage{amssymb}
\usepackage{amsmath}
\usepackage{mathrsfs}
\usepackage{amsthm}
\usepackage{graphics}
\usepackage{color}
\usepackage{dsfont}
\theoremstyle{definition} \newtheorem{lemma}{Lemma}[section]
\theoremstyle{definition} 
\theoremstyle{definition} \newtheorem{theorem}[lemma]{Theorem}
\theoremstyle{definition} 
\theoremstyle{remark} 
\theoremstyle{definition} \newtheorem{corollary}[lemma]{Corollary}
\theoremstyle{definition} \newtheorem{proposition}[lemma]{Proposition}
\theoremstyle{remark} \newtheorem*{notation}{Notation}
\theoremstyle{remark} \newtheorem*{acknowledgements}{Acknowledgement}
\numberwithin{lemma}{section}
\newcommand{\R}{\mathbb{R}}
\newcommand{\cl}{\overline}
\newcommand{\one}{\log}
\newcommand{\two}{\log \log}
\newcommand{\three}{\log \log \log}
\newcommand{\tw}{\tilde{w}}
\newcommand{\tp}{\tilde{\phi}}

\newcommand{\tu}{\tilde{u}}

\newcommand{\tl}{\tilde{l}}
\newcommand{\tL}{\tilde{L}}
\newcommand{\tJ}{\tilde{J}}
\newcommand{\scrl}{\mathscr{L}}
\newcounter{count}
\numberwithin{count}{lemma}
\newcounter{psi}
\newcounter{T}
\numberwithin{T}{lemma}
\renewcommand{\theT}{(T:\arabic{T})}
\newcounter{R}
\numberwithin{R}{lemma}
\renewcommand{\theR}{(R:\arabic{R})}

\numberwithin{wnconstruction}{lemma}
\renewcommand{\thepsi}{($\psi$:\arabic{psi})}
\newcommand{\beq}{\begin{equation}}
\newcommand{\eeq}{\end{equation}}
\newcommand{\bq}{\begin{equation*}}
\newcommand{\eq}{\end{equation*}}
\begin{document}

\title[A continuous Lagrangian with singular minimizer]{A one-dimensional variational problem with continuous Lagrangian and singular minimizer
}


\author{Richard Gratwick \and David Preiss       
}


\address{
              Mathematics Institute, Zeeman Building, University of Warwick, Coventry, CV4 7AL, UK. \\
}
              \email{R.T.Gratwick@warwick.ac.uk,\ D.Preiss@warwick.ac.uk}           

\date{9th April 2010}

\begin{abstract}
We construct a continuous Lagrangian, strictly convex and superlinear in the third variable, such that the associated variational problem has a Lipschitz minimizer which is non-differentiable on a dense set.  More precisely, the upper and lower Dini derivatives of the minimizer differ by a constant on a dense (hence second category) set.  In particular, we show that mere continuity is an  insufficient smoothness assumption for Tonelli's partial regularity theorem.
\end{abstract}
\maketitle
\section{Introduction}
\label{intro}
The problem of minimizing the one-dimensional variational integral
$$
\mathscr{L}(u) = \int_a^b L(t, u(t), u'(t))\, dt
$$
for some function $L \colon [a,b] \times \R \times \R \to \R$, $L \colon (t,y,p) \mapsto L(t, y, p)$, called the \emph{Lagrangian}, on a fixed bounded interval $[a,b]$ of the real line, over the class of absolutely continuous functions $u \colon [a,b] \to \R$ with prescribed boundary conditions, is now well understood.  The basic assumptions on $L$ for existence of such a minimizer are superlinearity and convexity in $p$, and minimal continuity assumptions.  This analysis was first performed by Tonelli~\cite{Ton}.
Our interest is partial regularity, on which the central result is again by Tonelli: under the assumptions that $L$ is $C^3$ and we have the slightly stronger strict convexity assumption $L_{pp} > 0$, we obtain \emph{partial regularity} of any minimizer $u \in \mathrm{AC}[a,b]$.  That is, the classical derivative of $u$ exists everywhere, with possibly infinite values, and the derivative is continuous as a map into the extended real line.  Thus the \emph{singular set}, the  set $E \subseteq [a,b]$ of points where the derivative is infinite,  is closed (and necessarily of course Lebesgue null); moreover off $E$ the minimizer $u$ inherits as much regularity as $L$ permits, i.e. $u$ is $C^k$ if $L$ is $ C^k$ for $k \geq 3$.    For a proof, see e.g. Ball and Mizel~\cite{BM}.  The book~\cite{BGH} gives a good summary of the results on  existence and partial regularity.

The most natural next question is to ask what we can know about the singular set $E$.  That minimizers of variational problems can have infinite derivative has been known since the paper of Lavrentiev~\cite{Lav}.  This presented the celebrated \emph{Lavrentiev phenomenon}, whereby when restricting the above minimization problem to even a dense subclass of the absolutely continuous functions (e.g. $C^1$ functions), the minimum value is \emph{strictly larger} than that minimum value taken over all absolutely continuous functions.  Mani\`a~\cite{Mania} gave an example of a polynomial Lagrangian superlinear in the third variable which exhibits the same phenomenon.  In such examples, the minimizer over the absolutely continuous functions has non-empty singular set $E$; Mani\`a's example has minimizer $t^{1/3}$ over domain $[0,1]$, thus $E= \{0\}$.  However, these examples do not satisfy the precise assumptions of the Tonelli partial regularity theorem, since the condition $L_{pp} > 0$ on the Lagrangian $L$ is violated (both the Lavrentiev and Mani\`a examples have $L_{pp} \geq 0$).  Thus the question of whether under the exact original conditions of the theorem, the set $E$ can be non-empty, is not answered by these examples.  However, Ball and Mizel~\cite{BM} modified Mani\`a's example to construct Lagrangians satisfying the conditions for the partial regularity theorem, i.e. in particular $L_{pp} > 0$, but with minimizers for which $E$ is non-empty.  They construct examples where $E$ consists of an end-point of the domain, and another where $E$ contains an interior point;  in the latter case, the Lavrentiev phenomenon occurs.  Davie~\cite{Davie} showed that nothing more can be said about $E$ in general by constructing  for a given arbitrary closed null set $E$ a $C^{\infty}$ Lagrangian $L$, superlinear in $p$ and with $L_{pp} > 0$, such that any minimizer (and at least one minimizer exists by Tonelli's existence result) has singular set precisely~$E$.

Some work has been done on lowering the smoothness assumptions in the partial regularity theorem.  Clarke and Vinter~\cite{CV85} prove a version of Tonelli's result under the assumptions of strict convexity and superlinearity in $p$, but requiring just that $L$ is locally Lipschitz in $(y,p)$ uniformly in $t$, and that $s \mapsto L(s, u(t), p)$ is continuous for all $(t, p)$, where $u$ is the minimizer under consideration.  They also examine a range of conditions to move to full regularity.  Their setting is in fact the vectorial case, dealing with functions $u \colon [a,b]\to \R^n$.  This example of the Tonelli regularity result is a corollary of their vectorial regularity results.  Sych\"ev~\cite{Sych91,Sych92} proves versions of the result under the usual strict convexity assumption and the condition that $L$ is (locally) H\"older continuous (in all variables).  Cs\"ornyei et. al.~\cite{uss} derive the result under the condition that a local  Lipschitz condition in $y$ holds locally uniformly in the other variables $(t,p)$.

The present paper shows that some smoothness assumption stronger than mere continuity (even in all three variables) of $L$ is necessary to obtain  partial regularity.  The main result  is the following:
\begin{theorem}\label{main}
 Let $T = e^{-e}/2$.  Then there exists Lipschitz $w \in \mathrm{AC}[-T, T]$ and continuous Lagrangian $L \colon [-T, T] \times \R \times \R \to [0, \infty)$, superlinear  in $p$ and with $L_{pp} > 0$,  such that
\begin{itemize}
 \item $w$ minimizes the associated variational problem
$$
\mathrm{AC}[-T, T] \ni u \mapsto \mathscr{L}(u) = \int_{-T}^{T} L(t, u(t), u'(t)) \,dt,
$$
over those $u \in \mathrm{AC}[-T, T]$ with $u(\pm T) = w(\pm T)$; but
   \item for dense $G_{\delta}$ (and hence second category) set $\Sigma \subseteq [-T, T]$, we have $ x \in \Sigma$ implies
$$
\overline{D}w(x) \geq 1\  \textrm{and}\ \underline{D}w(x) \leq - 1.
$$
\end{itemize}
\end{theorem}
\begin{notation} We shall write $\mathrm{AC}[a,b]$ for the class of absolutely continuous functions on a closed bounded interval $[a,b] \subseteq \R$.  One can of course also think of these as (representatives from the equivalence classes of) the Sobolev functions $W^{1,1}[a,b]$.
For $f \colon \R \to \R$, we write
$$
\mathrm{Lip}(f) = \sup_{\substack{s, t, \in X \\ s \neq t}}\frac{|f(s) - f(t)|}{|s-t|}.
$$
Although of course not true in general, this will always be a finite number in our usage.  The upper and lower Dini derivatives of a function $u \in \mathrm{AC}[a,b]$ at a point $x \in [a,b]$ are given by
$$
\overline{D}u(x) = \limsup_{t \to x} \frac{u(t) - u(x)}{t-x},\quad \textrm{and}\quad \underline{D}u(x) = \liminf_{t \to x}\frac{u(t) - u(x)}{t-x}.
$$
\end{notation}
\section{The construction}

Given any sequence of points in $(-T, T)$, we can construct a Lagrangian $L$ and minimizer $w$ with the set of non-differentiability points of $w$ containing this sequence.  The construction is essentially inductive, and hinges on the fact that a certain function $\tw$  is non-differentiable at one point, but minimizes a  continuous Lagrangian.  This basic Lagrangian is of form $(t, y, p)  \mapsto \tp(t, y - \tw(t)) + p^2$ for a ``weight function'' $\tp \colon [-T, T] \times \R \to [0, \infty)$ which penalizes functions which stray from $\tw$.  That is, $\tp(t, 0) = 0$, and for $|y| \leq |z|$ we have $0 \leq \tp(t,y) \leq \tp(t, z)$, for all $t \in [-T,T]$.  This summand of the Lagrangian then takes minimum value along the graph of $\tw$, and assigns larger values to functions $u$ the further their graph lies from that of $\tw$.  This immediately gives us a one-point example of non-differentiability of a minimizer, which already suffices to provide a counter-example to any Tonelli-like partial regularity result.  Additional points of non-differentiability are included by inserting translated and scaled copies of $\tw$ into the original $\tw$, and  passing to the limit, $w$, say.  The final Lagrangian is of form $(t,y,p) \mapsto \phi(t, y - w(t)) + p^2$, where $\phi$ is a sum of translated and truncated copies $\tp_n$ of $\tp$, each of  which penalizes functions which stray from $w$ in a neighbourhood of one of the points  $x_n$ in our given sequence.  We observe that many of the technicalities of the following proof are related to guaranteeing convergence of $w$ and $L$, and are in some sense secondary to the main points of the proof.

Define $\tw\colon \R \to \R$ by
\bq
\tw(t) =
\begin{cases}
t \sin \three 1/|t| & t \neq 0 \\
0 & t= 0,
\end{cases}
\eq
so
\beq
\tw \in C^{\infty}(\R \backslash \{0\}).\label{alpha}
\eeq
Note for $t \neq 0$,
\beq
\tw'(t) = \sin \three 1/|t| - \frac{\cos \three 1/|t|}{(\two 1/|t|)( \one 1/|t|)},\label{lamda}
\eeq
and we observe of course that this is an even function.
Also note that for $t \neq 0$,
\begin{align*}
|\tw ''(t)| & \leq \frac{1}{|t|(\two 1/|t|)(\one1/|t|)}\left( 1+ \frac{ \left(2 + \two 1/|t|\right)}{(\two 1/|t|)(\one 1/|t|)}\right)
\end{align*}
and hence see that
\begin{equation}
(t)|\tw''(t)| \to 0 \ \textrm{as}\  0 < |t| \to 0. \label{wn''cont}
\end{equation}
The following functions give us for each $t \in [-T, T]$ the exact coefficients  we shall eventually need in our weight function $\tp$.  Define $\psi^1, \psi^2 \colon \R \to [0, \infty)$ by
$$
\psi^1 (t) =
\begin{cases}
\frac{402}{|t| \two (1/  5|t|)} &  t \neq 0 \\
0 & t = 0
\end{cases}
\quad\textrm{and}\quad \psi^2 (t)  =
\begin{cases} 3 + 4|w''(t)| & t \neq 0 \\
0 & t=0,
\end{cases}
$$
and so define $\psi \colon \R \to [0, \infty)$ by $\psi(t) = \psi^1(t) + \psi^2(t)$.   Note that by~\eqref{alpha} and~\eqref{wn''cont}
\begin{list}{\thepsi}{\usecounter{psi}}\item  $\psi \in C(\R \backslash \{0\})$; and\label{psi1}
\item
$ t \mapsto t\psi(t)$ defines a function in $ C(\R)$, with value 0 at 0\label{psi2}.
\end{list}
Define $C> 0$ by
\beq
C := 1 + \!\sup_{t \in [-T, T]}\! 5|t| \psi (t), \label{C}
\eeq
so~\ref{psi2} guarantees $C < \infty$.

Let  $\{ x_n\}_{n=0}^{\infty}$ be a sequence of points in $(-T,T)$.  For notational convenience, we assume $x_0 = 0$.   By our choice of $T$, we have for all $n \geq0$ and $t \in [-T, T]\backslash\{x_n\}$ that
 \begin{align}
 \frac{1}{\one 1/|t-x_n|}
 &\leq e^{-1}; \label{i}\ \\ \frac{1}{\two 1/ |t-x_n|}
 & \leq 1; \label{one}\ \textrm{and}\\   \frac{1}{\two 1/|t-x_n|}& \geq |t- x_n|.\label{tau}
 \end{align}

For each $n \geq 1$, we write
$$
\sigma_n = \min_{0 \leq i < n} |x_i - x_n| / 2 > 0.
  $$
For each $n \geq 0$ we now define the translated functions $\tw_n \colon [-T,T] \to \R$  by \linebreak $\tw_n(t) =  \tw (t- x_n)$ and $\psi_n \colon [-T,T] \to [0,\infty)$ by $\psi_n(t)= \psi(t-x_n)$.

We want to construct a sequence of Lipschitz continuous functions $w_n$ with uniformly bounded Lipschitz constant, and with  $w_n = \tw_i$ on a neighbourhood of $x_i$, thus $w_n$ is singular at $x_i$, for each $0 \leq i \leq n$.  We first define a decreasing sequence $T_n \in (0,1)$ and hence intervals $Y_n := [x_n - T_n, x_n + T_n]$.  In the inductive construction of $w_n$ we shall modify $w_{n-1}$ only on $Y_n$.

Define a sequence of constants $K_n \geq 1 $ by setting $K_0 = 1$ and so that for $n\geq 1$,
\begin{gather}
K_n \geq 1+ K_{n-1} \label{Kn2};\ \textrm{and}\\
2 \sum_{i=0}^{n-1} |\tw''_i (t)| \leq K_n\ \textrm{for $t \in [-T,T]$ such that $|x_i - t| \geq \sigma_n$ for all $0 \leq i \leq n-1$.}\label{Kn}
\end{gather} This is possible for $K_n < \infty$ by~\eqref{alpha}.

Let $T_0 = T$, so $Y_0 = [-T,T]$.  For each $n\geq 1$ we inductively define $T_n \in (0,1)$ small enough such that $Y_n := [x_n - T_n, x_n + T_n] \subseteq [-T,T]$, and the following conditions hold:
\begin{list}{\theT}{\usecounter{T}}
\item $T_n < \sigma_n$; \label{t2}
\item $T_n < T_{n-1}/2 $; \label{t3}
\item $|(t-x_n) \psi_n (t)| < 2^{-n} / 5 $ for $t \in Y_n$; and\label{t4}
\item $ T_n < K_n^{-1}$\label{t5}.
\end{list}
Note that~\ref{t4} is possible by~\ref{psi2}.  Since we only modify $w_{n-1}$ on $Y_n$ to construct $w_n$, we only need to add more weight to our Lagrangian for $t \in Y_n$.  Recalling that we are always working with translations of the same basic function $\tp$ (which we will define explicitly later), we know that we can choose the intervals $Y_n$ small enough so that summing all the extra ``weights'' we need, we still converge to a continuous function.
That the intervals of modification are small enough in this sense is the reason behind conditions~\ref{t3} and~\ref{t4}.  Since $T_0<1$,~\ref{t3} guarantees in particular that
\beq
T_n < 2^{-n}\ \textrm{for all $n\geq0$}.\label{iv}
\eeq
Condition~\ref{t2} guarantees that the points in $Y_n$ are far away from the previous~$x_i$:
\beq
|x_i - t| > \sigma_n \ \textrm{for}\ 0 \leq i <n, \ \textrm{whenever}\ t \in Y_n; \label{Ynfarfromxi}
\eeq
this stops the subintervals we later consider from overlapping.  Condition~\ref{t5} just simplifies some estimates.

We emphasize that this sequence $\{T_n\}_{n=0}^{\infty}$ is constructed independently of the later constructed $w_n$; the inductive construction of these functions will require us to pass further down the sequence of $T_n$ than induction would otherwise allow, as we now see.

For $n \geq 0$, find $m_n > n$ such that
\beq
2^{-m_n} < \frac{T_{n+1}^2}{256}\label{mn}.
\eeq
Choose an open cover  $G_n \subseteq [-T,T]$ of the points $\{x_i\}_{i=0}^{m_n}$
such that
\beq
\mathrm{meas}(G_n) \leq \frac{T_{n+1}^2 }{16C} \label{Fnmeas}.
\eeq
Now, by~\ref{psi1} we can find $1 < M_n < \infty$ such that we have
\beq
\sum_{i=0}^{m_n}\left(\max\{\psi_i (t), \psi_i (x_i + T_i)\}\right) \leq M_n \ \textrm{whenever $ t \in [-T,T] \backslash G_n$}\label{Mn}.
\eeq

Let $\epsilon_n = 2^{-n}(1- e^{-1})$.  Let $R_0 = T_0$ and for $n\geq 1$ inductively construct a decreasing sequence  $R_n \in (0,T_n]$ such that:\begin{list}{\theR}{\usecounter{R}}
\item $\frac{1}{(\two 1/R_n)(\one 1/R_n)} < \epsilon_n /2$\label{r1};
\item $R_n  < R_{n-1}/2$\label{r2}; and
\item $R_n < \frac{2^{-n}T_n^3 \epsilon_n}{128 \cdot 25M_{n-1}}$
.\label{r3}
\end{list}
Now define subintervals  $Z_n : = [x_n - R_n, x_n + R_n]$ of $Y_n$.  These intervals are those on which we aim to insert a copy of $\tw_n$ into $w_{n-1}$.  The $Z_n$ must be a very much smaller subinterval of $Y_n$ to allow the estimates we require to hold; the point of this stage in the construction is that we now let the derivative of $w_n$ oscillate on $Z_n$, so we have to make the measure of this set very small to have any control over the convergence.

\begin{lemma}\label{wnlemma}
There exists a sequence of $w_n \in \mathrm{AC}[-T,T]$ satisfying, for $ n \geq 0$:
\begin{list}{(\thecount)}{\usecounter{count}}
\item $w_n (t)= \alpha_n \tw_n(t) + \beta_n$ when $t \in [x_n - \tau_n, x_n + \tau_n]$, for some \linebreak
$\tau_n \in (0, R_n]$, some $\alpha_n \in [1,2)$, and some $\beta_n \in \R$;\label{wn=twn}
\item $w_n'$ exists and is locally Lipschitz on  $[-T,T] \backslash  \{x_i\}_{i=0}^n$;\label{wn'lip}
\item $|w_n' (t)| < 2 - \epsilon_n$ for $t \notin \{x_i\}_{i=0}^n $;\label{lipwn}
\item $|w_n''| \leq K_{n+1}$ on $Y_{n+1}$ almost everywhere;
\label{wn''bd}
\end{list}
and for $n \geq 1$:
\begin{list}{(\thecount)}{\usecounter{count}}\setcounter{count}{4}
\item $w_n = w_{n-1}$ off $Y_n$;\label{wn=wn-1}
\item $\|w_n - w_{n-1}\|_{\infty} < 10R_n$; \label{cvg}
\item $w_n (x_i) = w_{n-1} (x_i)$ for all $0 \leq i \leq n$; \label{wnfixxi}
\item $|w_n' (t) - w_{n-1}'(t)| < \frac{T_n^2}{128}$ for $t \notin Z_n \cup \{x_i\}_{i=0}^n $; \label{cvg'} and
\item $|w_n''(t)| < |w_{n-1}''(t)| + 2^{-n}$ for almost every $ t\notin  [x_n - \tau_n, x_n + \tau_n]$\label{wn''-wn-1''}.
    \end{list}
\end{lemma}
\begin{proof}
We easily check that defining $w_0 = \tw_0$ satisfies all the required conditions.  Condition~\eqref{wn=twn} is trivial for $\tau_0 = T_0$, $\alpha_0 = 1$, and $\beta_0 = 0$; and~\eqref{wn'lip} follows from~\eqref{alpha}.  Condition~\eqref{lipwn} follows from~\eqref{lamda},~\eqref{one}, and~\eqref{i} since for $t \neq x_0$ we have
$$
|w_0'(t)|  \leq 1 + \frac{1}{(\two 1/|t|)( \one1/|t|)}  \leq1 + \frac{1}{\one 1/|t|}  \leq 1 + e^{-1}
  =  2 - \epsilon_0.
  $$
  Condition~\eqref{wn''bd} follows from \eqref{Ynfarfromxi} and~\eqref{Kn}.

 Suppose for $n \geq 1$ we have constructed $w_i$ as claimed for all $0 \leq i <n$.  We demonstrate how to insert a certain scaled copy of $\tw_n$ into $w_{n-1}$.

Condition \ref{t2} implies that $x_i \notin Y_n$ for all $0 \leq i <n$, thus $w_{n-1}'$ exists and is Lipschitz on $Y_n$ by inductive hypothesis~\eqref{wn'lip}.  Define $m: = w_{n-1}'(x_n)$, so \linebreak$|m| < 2 - \epsilon_{n-1}$ by inductive hypothesis~\eqref{lipwn}.  (We introduce in this proof a number of variables, e.g. $m$,  which only appear in this inductive step.  Although they do of course depend on $n$, we do not index them as such, since they are only used while $n$ is fixed.) On some yet smaller subinterval $[x_n - \tau_n, x_n + \tau_n]$ of $Z_n$ we aim to replace $w_{n-1}$ with a copy of $\tw_n$, connecting this with $w_{n-1}$ off $Y_n$ without increasing too much either the first or second derivatives, hence the choice of $R_n$ as very much smaller than $T_n$.  Moreover we want to preserve a continuous first derivative.  Hence we displace $w_{n-1}$ by a $C^1$ function---dealing with either side of $x_n$ separately---so that on either side we approach $x_n$ on an affine function of gradient $m$ (a different function either side, in general), which we then connect up with $\tw_n$ at a point where $\tw_n'=m$.  Because we need careful control over the first and second derivatives, it is easiest to construct explicitly the cut-off function we in effect use.

A slight first problem is that so small might be the interval on which we consider $\tw_n$, the derivative might never be large enough in magnitude to perform the join described above.  Hence the possible need to scale $\tw_n$ up slightly by some number $\alpha_n \in (1,2)$ to ensure we can find points where the derivatives can agree.

If $|m| \leq 1$, then by continuity of $\tw_n'$ it is trivial that there exists $\tau_n \in (0, R_n]$ such that $\tw_n' (x_n - \tau_n) = m =  \tw_n'(x_n + \tau_n)$.  So no scaling is required, set $\alpha_n =1$.

If $|m| > 1$, in general we have to scale $\tw_n$ up slightly.  Let $A = \sup_{[x_n - R_n, x_n)} \tw_n'$, and $B = \inf_{[x_n - R_n, x_n)} \tw_n'$.  Then by~\eqref{lamda} and~\ref{r1}
$$
1 < A  \leq  1+ \frac{1}{\two 1/R_n \one 1/ R_n}  < 1 + \epsilon_n / 2
$$
and similarly $ -(1 + \epsilon_n / 2) < B < -1$.  Let $\rho = \min \{ |A|, |B|\}$, so $ 1 < \rho < 1 + \epsilon_n /2$.  These values are attained, say $\tw_n' (y) = A$ and $\tw_n'(z) = B$ for $y, z \in [x_n - R_n , x_n)$.  Thus we have $\tw_n'(y) = |\tw_n' (y)| \geq \rho$ and $-\tw_n' (z) = |\tw_n'(z)| \geq \rho$.  Put $\alpha_n = m / \rho$, so $|\alpha_n| < 2$.  Evidently the function $|\alpha_n \tw_n'|$ takes its maximum value over $[x_n - R_n, x_n)$ at $y$ or $z$, and so calculating
$$
|\alpha_n \tw_n'(y)|  <  \frac{|m|(1 + \epsilon_n / 2)}{ \rho}
 <  |m| (1 + \epsilon_n / 2)
<  |m| + \epsilon_n
 <  2 - \epsilon_{n-1} + \epsilon_n
 = 2 - \epsilon_n,
 $$
 and similarly for $|\alpha_n  \tw_n' (z)|$, we see $|\alpha_n \tw_n'| < 2 - \epsilon_n$ on $[x_n - R_n, x_n)$, and since this is an even function we have \beq |\alpha_n \tw_n'(t)| < 2 - \epsilon_n \ \textrm{for all $t \in Z_n \backslash \{x_n\}.$}\label{eta}\eeq We now show we have indeed scaled $\tw_n$ large enough, despite ensuring this bound.  If $m \geq 0$ we see that
 $$
 \alpha_n\tw_n' (y) = \frac{m \tw_n'(y)}{\rho} \geq m, \ \textrm{and}\ \alpha_n \tw_n' (z)= \frac{m\tw_n'(z)}{\rho} \leq -m \leq m,
 $$
 and if $m \leq 0$ we see that
 $$
 \alpha_n \tw_n' (y) = \frac{m \tw_n'(y)}{\rho} \leq m, \ \textrm{and}\ \alpha_n \tw_n'(z) = \frac{m\tw_n'(z)}{\rho} \geq -m \geq m.
 $$
 So in either case, since by~\eqref{alpha} $\tw_n'$ is continuous on $[x_n - R_n, x_n)$,  we can apply the intermediate value theorem to find $\tau_n \in (0, R_n]$ with $\alpha_n \tw_n' (x_n - \tau_n) = m$.  Thus also of course $\alpha_n \tw_n' (x_n + \tau_n) = m$.

We now construct the cut-off functions $\chi_l$ and $\chi_r$ we use on the left and right of $x_n$ respectively.  Additional constants and functions used in the construction are labelled similarly.

Let $\delta_l = m-w_{n-1}'(x_n - R_n)$.  So recalling that $w_{n-1}'$ is Lipschitz on $Y_n \supseteq Z_n$, we see by inductive hypothesis~\eqref{wn''bd} and~\ref{r3} that
\begin{align}
|\delta_l| = |w_{n-1}'(x_n) - w_{n-1}'(x_n - R_n)| \leq \|w_{n-1}''\|_{L^{\infty}(Z_n)} R_n &\leq K_n R_n \label{moddeltaa}\\
&< \epsilon_n.\label{moddeltab}
\end{align}
Define
$$
c_l = w_{n-1}(x_n) + \alpha_n \tw_n (x_n - \tau_n) - m (x_n - \tau_n) - w_{n-1}(x_n - R_n) + m (x_n -R_n).
$$
The point is that the function $t \mapsto mt + w_{n-1}(x_n - R_n) - m(x_n - R_n) + c_l$ is an affine function with gradient $m$ which takes value $w_{n-1}(x_n - R_n) + c_l$ at $(x_n - R_n)$ and value $m (x_n - \tau_n) + w_{n-1}(x_n - R_n) - m (x_n - R_n) + c_l = w_{n-1}(x_n) + \alpha_n \tw_n (x_n - \tau_n)$ at $(x_n - \tau_n)$.

Note that by inductive hypothesis~\eqref{lipwn},
\begin{align}
|c_l| & \leq  |\alpha_n \tw_n (x_n - \tau_n)| + |w_{n-1}(x_n) - w_{n-1}(x_n - R_n)| + |m| |(x_n - R_n) - (x_n - \tau_n)| \nonumber\\
& < |\alpha_n |\tau_n + 2R_n + 2R_n \nonumber\\
& < 6R_n. \label{gamma}
\end{align}
Now put $d_l = \frac{4}{T_n}(c_l - \frac{\delta_l}{2}(T_n/2  - R_n).$  Define piecewise affine $g_l \colon [-T,T] \to \nolinebreak\R$ by stipulating
$$
g_l(x_n - T_n) = 0 = g_l (x_n - T_n/2), \ g_l(x_n - 3T_n/4) = d_l,
$$
and
$$
g_l(t) =
\begin{cases}
0 & t \leq x_n - T_n \\
\delta_l & t \geq x_n - R_n\\
\textrm{affine} & \textrm{otherwise.}
\end{cases}
$$
So by definition of $d_l$,
\beq
\int_{-T}^{x_n - R_n} g_l(t)\,dt = \int_{x_n - T_n}^{x_n - R_n} g_l(t)\,dt = \frac{1}{2}\left(\frac{T_n d_l}{2} + (T_n/2 - R_n)\delta_l\right) = c_l.\label{intgl}
\eeq
Now, $\|g_l\|_{\infty} = \max\{ |\delta_l|, |d_l|\}$. We see by~\eqref{gamma},~\eqref{moddeltaa}, and~\ref{r3} that
\begin{align}
|d_l| \leq \frac{4}{T_n}\left(|c_l| + \frac{|\delta_l|}{2}(T_n/2-R_n)\right)& < \frac{4}{T_n}\left(6R_n + \frac{T_n K_n R_n}{4}\right)\nonumber \\
 &= \frac{24R_n}{T_n} + K_nR_n \label{iia}\\
&< \epsilon_n.\nonumber
\end{align}
So, comparing with~\eqref{moddeltaa} and~\eqref{moddeltab}, we have
\begin{align}
\|g_l\|_{\infty} &\leq \frac{24R_n}{T_n} + K_nR_n \label{epsilonlb}\\
&< \epsilon_n.\label{epsilonla}
\end{align}
Also, $g_l'$ exists almost everywhere and satisfies $\|g_l'\|_{\infty} = \max\{\frac{4 |d_l| }{T_n}, \frac{|\delta_l|}{T_n/2 - R_n}\}$, and so since, from~\eqref{iia} and~\ref{r3}
$$
\frac{4|d_l|}{T_n} < \frac{4}{T_n} \left(\frac{24R_n}{T_n} + K_n R_n\right) = \frac{96R_n}{T_n^2} + \frac{4K_n R_n}{T_n} < 2^{-n},
$$
and since~\ref{r3} in particular implies $R_n < T_n/4$, using~\eqref{moddeltaa} and~\ref{r3} we see $$
\frac{|\delta_l|}{(T_n / 2) - R_n} < \frac{4 R_n K_n}{T_n} < 2^{-n},
$$
we have
\beq
\|g_l'\|_{\infty} < 2^{-n}\label{zetal}.
\eeq

We can now define $\chi_l \colon [-T,T] \to \R$ by $\chi_l (t) = \int_{-T}^t g_l(s)\,ds$.  This gives  \linebreak$\chi_l \in C^1[-T,T]$ such that $\chi_l' = g_l$ everywhere, $\chi_l'' = g_l'$ almost everywhere, and, by~\eqref{intgl},
$$
\chi_l(x_n - T_n) = 0, \ \chi_l (x_n - R_n) = c_l, \ \chi_l'(x_n - R_n) = g_l(x_n - R_n) = \delta_l.
$$

We perform a very similar argument on the right of $x_n$, to construct piecewise affine function $g_r \colon [-T,T] \to \R$.  Define
$$
c_r = w_{n-1}(x_n) + \alpha_n \tw_n (x_n + \tau_n)   - m (x_n + \tau_n) - w_{n-1}(x_n + R_n) + m (x_n + R_n),
$$
and $\delta_r = m - w_{n-1}'(x_n + R_n)$,  and finally $d_r = \frac{4}{T_n}(c_r + \frac{\delta_r}{2}(T_n / 2 - R_n))$. Then again stipulate
$$
g_r (x_n + T_n/2) = 0 = g_r(x_n + T_n),\ g_r (x_n +3T_n/4) = -d_r,
$$
and elsewhere
$$
g_r(t)=
\begin{cases}
\delta_r & t\leq x_n + R_n \\ 0 & t \geq x_n + T_n \\ \textrm{affine} & \textrm{otherwise}.
\end{cases}
$$
So by definition of $d_r$, we have
\begin{equation}
\int_{x_n + R_n}^{x_n + T_n}g_r (t)\,dt= \frac{1}{2}\left(\delta_r (T_n/2 - R_n) - \frac{d_r T_n}{2}\right) =  -c_r.\label{gr}
\end{equation}
All the numbers $c_r, \delta_r, d_r$ satisfy the same bounds as their left-hand counterparts, and thus $g_r$ satisfies the same bounds as $g_l$ above, i.e.
\begin{align}
\|g_r\|_{\infty}  &\leq \frac{24R_n}{T_n} + K_n R_n \label{epsilonrb}\\
 &< \epsilon_n \label{epsilonra}
 \end{align}
 and
 \begin{equation}
 \|g_r'\|_{\infty}  <  2^{-n}\label{zetar}.
 \end{equation}
 We now define $\chi_r \colon [-T,T] \to \R$ by
 $$
 \chi_r (t)= c_r - \delta_r((x_n + R_n) - (-T))+ \int_{-T}^t g_r(s)\,ds,
 $$
 which gives $\chi_r \in C^1[-T,T]$ such that $\chi_r' = g_r $ everywhere, $\chi_r'' = g_r'$ almost everywhere, and, by~\eqref{gr},
 $$
 \chi_r (x_n + R_n) = c_r, \ \chi_r (x_n + T_n) = 0, \ \chi_r'(x_n + R_n) = g_r (x_n + R_n) = \delta_r.
$$

We can now define $w_n \colon [-T,T] \to \R$ by
$$
w_n(t) =
\begin{cases}
w_{n-1}(t) + \chi_l (t) & t \leq x_n - R_n \\
mt + w_{n-1}(x_n - R_n) - m(x_n - R_n) + c_l & x_n - R_n < t <x_n - \tau_n \\
\alpha_n \tw_n (t) + w_{n-1}(x_n) & x_n - \tau_n \leq t \leq x_n + \tau_n \\
mt + w_{n-1}(x_n + R_n) - m (x_n + R_n) + c_r & x_n + \tau_n < t < x_n + R_n \\
w_{n-1}(t) + \chi_r (t) & x_n + R_n \leq t .
\end{cases}
$$
We see $w_n$ is continuous by construction.  Condition~\eqref{wn=twn} is immediate, with $\alpha_n$ and $\tau_n$ as defined, and $\beta_n = w_{n-1}(x_n)$.  We note that since $\chi_l(t) = 0$ for $t < x_n - T_n$, $\chi_r(t) = 0$ for $t > x_n + T_n$, we have that $w_n = w_{n-1}$ off $Y_n$, as required for \eqref{wn=wn-1}.

We see that $w_n'$ exists off $\{x_i\}_{i=0}^n$ by inductive hypothesis~\eqref{wn'lip},~\eqref{alpha}, and by construction,  and is given by
$$
w_n'(t) =
\begin{cases}
w_{n-1}'(t) + g_l (t) & t \leq x_n - R_n \\
m & x_n -R_n < t < x_n - \tau_n \\
\alpha_n \tw_n'(t)& x_n - \tau_n \leq t < x_n, \ x_n < t \leq x_n + \tau_n \\
m & x_n + \tau_n < t < x_n + R_n \\
w_{n-1}'(t) + g_r (t) & x_n + R_n \leq t.
\end{cases}
$$
This is locally Lipschitz on $[-T,T] \backslash \bigcup_{i=0}^n\{x_i\}$ by inductive hypothesis~\eqref{wn'lip} on $w_{n-1}'$,~\eqref{alpha}, and since $g_l$ and $g_r$ are Lipschitz.  By inductive hypothesis~\eqref{lipwn}, and conditions~\eqref{epsilonla},~\eqref{epsilonra}, and~\eqref{eta}, we have for $ t \notin \{x_i\}_{i=0}^n$,
$$
|w_n'(t)| \leq
\begin{cases}
|w_{n-1}'(t)| + |g_l(t)| < 2 - \epsilon_n & t \leq x_n - R_n\\
|m| < 2 - \epsilon_n & x_n - R_n < t < x_n - \tau_n \\
|\alpha_n \tw_n'(t)| < 2 - \epsilon_n & x_n - \tau_n \leq t < x_n, \ x_n < t \leq x_n + \tau_n \\
|m| < 2 - \epsilon_n & x_n + \tau_n < t < x_n + R_n \\
|w_{n-1}'(t)| + |g_r(t)| < 2- \epsilon_n & x_n + R_n \leq t.
\end{cases}
$$
Hence~\eqref{lipwn}.  We also see by~\eqref{epsilonlb} and~\ref{r3} that for $ t \leq x_n - R_n$, $t \notin \{x_i\}_{i=0}^{n-1}$,
$$
|w_n' (t) - w_{n-1}'(t)| = |g_l(t)| \leq \frac{24R_n}{T_n} + K_n R_n < \frac{T_n^2}{128};
$$
and similarly for $ t \geq x_n + R_n$, $ t \notin \{x_i\}_{i=0}^{n-1}$, by~\eqref{epsilonrb} and~\ref{r3}  we have that
$$
|w_n'(t) - w_{n-1}'(t)| = |g_r(t)| \leq \frac{24R_n}{T_n} + K_n R_n < \frac{T_n^2}{128};
$$
hence \eqref{cvg'}.
Also $w_n''$ exists almost everywhere and where it does, is given by

$$
w_n''(t) =
\begin{cases}
w_{n-1}''(t) + g_l'(t) & t < x_n - R_n \\
0 & x_n - R_n < t < x_n - \tau_n \\
\alpha_n \tw_n''(t) & x_n - \tau_n < t < x_n, \ x_n < t <x_n + \tau_n \\
0 & x_n + \tau_n < t < x_n + R_n\\
w_{n-1}''(t) + g_r'(t) & x_n + R_n < t
\end{cases}
$$
and thus by~\eqref{zetal}, for $ t < x_n - R_n $ we have
$$
|w_n''(t)| \leq |w_{n-1}''(t)| + |g_l'(t)| < |w_{n-1}''(t)| + 2^{-n},
$$
and by~\eqref{zetar}, for $x_n + R_n < t$, we have
$$
|w_n''(t)|  \leq |w_{n-1}''(t)| + |g_r'(t)| < |w_{n-1}''(t)| + 2^{-n}.
$$
 Hence~\eqref{wn''-wn-1''}.  We now check \eqref{wn''bd}.  Let $ t \in Y_{n+1}$.  Then by~\eqref{Ynfarfromxi} we see that
$$
2\sum_{i=0}^n |\tw_i''(t)| \leq K_{n+1}
$$
precisely by choice of $K_{n+1}$ in~\eqref{Kn}.  Let $0 \leq k \leq n$ be such that $t \in Y_k \backslash \bigcup_{i=k+1}^n Y_i$.    Then by inductive hypothesis~\eqref{wn=wn-1} for $k+1, \dots, n$ (we have checked this for $k=n$), we have that $w_n = w_k$ on a neighbourhood of $t$, so $w_n''(t) = w_k ''(t)$ where both sides exist, i.e. almost everywhere.  If $ t \notin [x_k - \tau_k , x_k + \tau_k]$, then by inductive hypotheses~\eqref{wn''-wn-1''} (we have checked this for $k=n$) and~\eqref{wn''bd}, and by~\eqref{Kn2}, we have almost everywhere,
$$
|w_n''(t)| = |w_k''(t) | \leq |w_{k-1}''(t)| + 2^{-k} \leq K_{k} + 1  \leq K_{n+1}
$$
as required.  If $ t \in (x_k - \tau_k , x_k + \tau_k)$, then by inductive hypothesis~\eqref{wn=twn} (we have checked this for $k=n$), almost everywhere we have, as noted above,
$$
|w_n''(t)| = |w_k''(t)| = |\alpha_k \tw_k ''(t)| < 2 |\tw_k''(t)| \leq 2\sum_{i=0}^k |\tw_i ''(t)| \leq 2\sum_{i=0}^n |\tw_i''(t)| \leq  K_{n+1}
$$
as required.

Now observe that on $[-T, x_n - R_n]$, we have, by definition, and using~\eqref{iia},~\eqref{moddeltaa}, and~\ref{t5}, that
\begin{align*}
|\chi_l| & \leq  \frac{1}{2} \left(\frac{T_n}{2} |d_l| + (T_n/2-R_n)|\delta_l|\right) \\
&\leq  \frac{T_n}{4}\left(\frac{24R_n}{T_n} + K_nR_n + K_n R_n \right)\\
& \leq R_n\left(6 + \frac{ T_n K_n }{2}\right)\\
& < 7R_n.
\end{align*}
A similar estimate holds for $\chi_r$ on $[ x_n + R_n, T]$: we note first by~\eqref{gr} that 
\begin{align*}
\chi_r(t) & =  c_r - \delta_r((x_n + R_n) + T) + \int_{-T}^t g_r(s)\, ds \\
& =  c_r + \int_{x_n + R_n}^t g_r(s)\,ds \\
& =  - \int_{x_n + R_n}^{T} g_r(s)\,ds + \int_{x_n + R_n}^t g_r(s)\,ds \\
& =  - \int_t^{ T} g_r(s)\,ds
\end{align*}
and then, since $|\chi_r| \leq \int_{x_n + R_n}^{T}|g_r|$ on $[x_n + R_n, T]$,  we can estimate as above.
So, for  $x_n - T_n \leq t \leq x_n - R_n$, we have
$$
|w_n (t) - w_{n-1}(t)| = |\chi_l(t)| \leq 7R_n
$$
and similarly for $x_n + R_n \leq t \leq x_n + T_n$ we have
$$
|w_n (t) - w_{n-1}(t)| = |\chi_r(t)| \leq 7R_n.
$$
By inductive hypothesis~\eqref{lipwn} and~\eqref{gamma}, we have for $x_n - R_n < t < x_n - \tau_n $ that
$$
|w_n (t) - w_{n-1}(t)| \leq |mt - m(x_n - R_n)| + |w_{n-1}(x_n - R_n) - w_{n-1}(t)| + |c_l| < 10R_n
$$
and similarly for $x_n + \tau_n < t < x_n + R_n$ we have
$$
|w_n (t) - w_{n-1}(t)|\leq |mt - m(x_n +R_n)| + |w_{n-1}(x_n +R_n) - w_{n-1}(t)| + |c_r| <10R_n.
$$
Finally for $x_n - \tau_n \leq t \leq x_n + \tau_n $, by inductive hypothesis~\eqref{lipwn} again we have
$$
|w_n (t) - w_{n-1}(t)| \leq |\alpha_n \tw_n (t)| + |w_{n-1}(x_n) - w_{n-1}(t)| \leq 2|\tau_n| + 2|\tau_n| \leq 4R_n .
$$
Hence we have, using also~\eqref{wn=wn-1} (which we have checked for $n$),
$$
\|w_n - w_{n-1}\|_{\infty} = \sup_{t \in Y_n}|w_n (t) - w_{n-1}(t)| < 10R_n
$$
as required for~\eqref{cvg}.

We finally check~\eqref{wnfixxi}.  Let $ 0 \leq i \leq n$.  If $ i < n$, then $x_i \notin Y_n$ by \ref{t2}, so 
$w_n (x_i) = w_{n-1}(x_i) $ by~\eqref{wn=wn-1}.  Since $\tw_n (x_n) = 0$, we see  from the construction that  $w_n (x_n) = w_{n-1}(x_n)$  as required for the full result.
\end{proof}
We now show easily that this sequence converges to a Lipschitz function $w$.  This $w$ will be our singular minimizer.

\begin{lemma}\label{wncvg}
The sequence $\{w_n\}_{n=0}^{\infty}$ converges uniformly to some $w \in \mathrm{AC}[-T,T]$ such that
\begin{list}{(\thecount)}{\usecounter{count}}
\item  $\mathrm{Lip}(w) \leq 2$; \label{lipw}
\item for all $n \geq 0$, $w(x_i) = w_n (x_i)$ for all $0 \leq i \leq n+1$; \label{wfixxi}
\item for all $n\geq 0$, $w' = w_n'$ almost everywhere off $\bigcup_{i=n+1}^{\infty}Y_i$; \label{wn'tow'}and
\item $ \|w- w_n\|_{\infty} \leq 20R_{n+1}$ for all $ n\geq 0$.\label{w-wn}
\end{list}
\end{lemma}
\begin{proof}
Let $n\geq 0$.  We use~\eqref{cvg} and~\ref{r2}  to see that
for $m > n$ we have
$$
\|w_m - w_n\|_{\infty} < 10 (R_m + \dots + R_{n+1}) \leq 10 (2^{-(m - (n+1))} + \dots + 1)R_{n+1} < 20R_{n+1}.
$$
Hence the sequence $\{w_n\}_{n=0}^{\infty}$ is uniformly Cauchy, and therefore converges uniformly to some $w \in C[-T,T]$.  Condition~\eqref{w-wn} follows immediately, and~\eqref{lipw} follows from~\eqref{lipwn}, and so of course certainly $w \in \mathrm{AC}[-T,T]$.  Condition~\eqref{wfixxi} follows directly from~\eqref{wnfixxi}.

We check~\eqref{wn'tow'}.  Fix $n \geq 0$, let $ t \in [-T,T] \backslash (\{x_i\}_{i=0}^n \cup \bigcup_{i=n+1}^{\infty}Y_i)$, and let  $j > n$.   In particular then $t \notin \bigcup_{i=n+1}^{j} Y_i$ which is a closed set, thus by~\eqref{wn=wn-1} there is a neighbourhood of $t$ on which $w_j = w_n$. Therefore $w_j'(t) = w_n'(t)$, which exists by~\eqref{wn'lip}.  So $\lim_{j \to \infty} w_j'(t)$ exists and equals $w_n'(t)$.

For each $t \in [-T,T] \backslash (\{x_i\}_{i=0}^{\infty} \cup \bigcap_{n=0}^{\infty} \bigcup_{i= n}^{\infty} Y_i )$, this argument runs for some $n \geq 0$. Since for all $n \geq 0$, by~\ref{t3},
 $$
 \mathrm{meas}\left(\bigcap_{n=1}^{\infty} \bigcup_{i=n}^{\infty} Y_i\right) \leq \mathrm{meas}\left(\bigcup_{i=n}^{\infty}Y_i\right) \leq \sum_{i=n}^{\infty}2T_i \leq 4T_n,
 $$
 and~\ref{t3} guarantees $T_n \to 0$ as $ n \to \infty$, we see that $w_n'$ has a pointwise limit almost everywhere.  We can easily see this limit must be equal to $w'$: for $t \in [-T,T]$, we recall from~\eqref{lipwn} that $|w_n'| \leq 2$ for all $n\geq 0$ and use the dominated convergence theorem to see
 $$
 \int_{-T}^t \lim_{n \to \infty} w_n'(s)\,ds  =  \lim_{n \to \infty} \int_{-T}^t w_n'(s)\, ds  =  w(t) - w(-T)
 $$
 and hence $w' = \lim_{n \to \infty} w_n'$ almost everywhere.  Since almost everywhere off \linebreak$\bigcup_{i=n+1}^{\infty}Y_i$ we have $\lim_{i \to \infty} w_i' = w_n'$ as shown above, we have the result claimed.
 \end{proof}
Our basic weight function $\tp \colon [-T,T] \times \R \to [0, \infty)$ will be given by
$$
\tp(t,y) =
\begin{cases}
0 & t = 0 \\
5 \psi(t)|t| & |y| \geq 5 |t| \\
\psi(t)|y| & |y| \leq 5 |t|.
\end{cases}
$$
We need some bound of form $|\phi(t,y)| \leq c |t| \psi(t)$ to ensure continuity of $\phi$; it turns out (see Lemma~\ref{lemma1}) that sensitive tracking of $|y|$ only for $|y| \leq 5|t|$ suffices in the proof of minimality.  Our function $\tw$ was constructed precisely so that~\eqref{wn''cont} and hence~\ref{psi2} hold, and hence that this $\tp$ is continuous.

We in fact will find it useful to split $\tp$ into the summands by which we defined~$\psi$.  Precisely, we define for each $n \geq 0$ our translated weight functions  \linebreak$\tp_n^1, \tp_n^2 \colon [-T,T] \times \R \to  [0, \infty)$ as follows.  For $n \geq 0$, and for $i = 1,2$,  we recall that we only need extra weight on $Y_n$, so define for $(t, y) \in Y_n \times \R$
$$
\tp_n^i (t,y) =
\begin{cases}
0 & t= x_n \\
5 \psi_n^i (t)|t-x_n | & |y| \geq 5|t-x_n| \\
\psi_n^i(t) |y| & |y| \leq 5|t-x_n|
\end{cases}
$$
and then just extend to a function on the whole of $[-T,T] \times \R$ by defining for $(t,y) \in ([-T,T] \backslash Y_n) \times \R$
$$
\tp_n^i(t,y) =
\begin{cases}
5\psi_n^i(x_n + T_n) T_n & |y| \geq 5T_n \\\psi_n^i(x_n +T_n) |y|& |y| \leq 5T_n.
\end{cases}
$$
For $n\geq 0$ we thus define  $\tp_n \colon [-T,T] \times \R \to [0, \infty)$ by  $\tp_n (t,y) = \tp_n^1 (t,y) + \tp_n^2 (t,y)$.  By \ref{psi2} we see that $\tp_n \in C([-T,T] \times \R)$
.

It is easily seen that for fixed $t \in [-T,T]$, for all $ n\geq 0$, we have
\begin{gather*}
\tp_n (t,y) \leq \tp_n(t, z)\ \textrm{whenever}\ |y| \leq |z|;\label{tphiinc}\\
\mathrm{Lip}(\tp_n(t,.)) \leq \max\{\psi_n (t), \psi_n (x_n + T_n)\};\label{liptphi}\ \textrm{and}\\
\tp_n (t, 0) = 0.\label{tp0is0}
\end{gather*}
 Defining $\phi_n \colon [-T,T] \times \R \to [0, \infty)$ by $\phi_n(t,y) = \sum_{i=0}^n \tp_i (t,y)$ gives a sequence of functions $\phi_n \in C([-T,T] \times \R)$ such that for each fixed $t \in [-T,T]$, for all $n \geq0$,
 \begin{gather}
 \phi_n (t, y) \leq \phi_n (t, z) \ \textrm{whenever $ |y| \leq |z|;$}\label{phininc}\\
 \mathrm{Lip}(\phi_n(t, .)) \leq \sum_{i=0}^n \left(\max\{\psi_i (t), \psi_i (x_i + T_i)\}\right);\label{lipphin}\ \textrm{and}\\
 \phi_n (t, 0 ) = 0.\label{phin0is0}
 \end{gather}
 For $n \geq 1$, by \ref{t4}, we see that for all $(t, y) \in [-T,T] \times \R$
 $$
 0 \leq \tp_n (t,y) \leq \sup_{t \in Y_n} 5 \psi_n(t)|t-x_n| < 2^{-n}.
 $$
 So defining  $\phi(t,y) = \sum_{i=0}^{\infty} \tp_i (t,y)$ gives  $\phi \in C([-T,T] \times \R)$ with, by~\eqref{C}, \beq
 \| \phi\|_{\infty} \leq \| \tp_0\|_{\infty} + \sum_{i=1}^{\infty}\|\tp_i\|_{\infty} \leq \|\tp_0\|_{\infty} + \sum_{i=1}^{\infty}2^{-i} = \|\tp_0\|_{\infty} + 1 = C,\label{phibd}
 \eeq
 and
 \beq
 \|\phi- \phi_n\|_{\infty} \leq \sum_{i=n+1}^{\infty} \|\tp_i\|_{\infty}< \sum_{i = n+1}^{\infty}2^{-i} = 2^{-n}.\label{phi-phin}
 \eeq
 By passing to the limit in the relations~\eqref{phininc} and~\eqref{phin0is0} we see that for fixed \linebreak$ t \in [-T,T]$,
\begin{gather}
 \phi(t, y) \leq \phi(t, z) \ \textrm{whenever $ |y | \leq |z|$}; \label{phiinc}\ \textrm{and}\\  \phi(t, 0) = 0. \label{phi0is0}
 \end{gather}
 We shall write $\phi = \phi^1 + \phi^2$ where $\phi^i = \sum_{j=0}^{\infty} \tp_j^i$ for $i=1,2$.

We can now define a continuous Lagrangian $L\colon [-T,T] \times \R \times \R \to [0, \infty)$, superlinear and strictly convex in $p$, by setting
$$
L(t, y, p) = p^2 + \phi(t, y - w(t)).
$$
Note in fact that $L$ is differentiable with respect to $p$ and $L_{pp}(t,y,p) = 2 > 0$ for all $(t, y, p) \in [-T, T] \times \R \times \R$, thus it does satisfy the stronger strict convexity assumption required by Tonelli.

Associated with this is the usual variational problem given by defining functional $\mathscr{L}\colon \mathrm{AC}[-T,T] \to [0, \infty)$ by
$$
\mathscr{L}(u) = \int_{-T}^{T} L(t, u(t), u'(t))\,dt
$$
and seeking to minimize $\mathscr{L}(u)$ over those functions $u \in \mathrm{AC}[-T,T]$ with boundary conditions $u(\pm T) = w(\pm T)$.  We shall refer to this set-up as ($\star$).

\section{Minimality}
 We shall find the following approximations of our functional $\mathscr{L}$ useful: for $ n\geq 0$ define $L_n \colon [-T,T] \times \R \times \R \to [0, \infty)$ by
 $$
 L_n(t, y, p) = p^2 + \phi(t, y - w_n (t)),
 $$
 and  define the corresponding functional $\mathscr{L}_n \colon \mathrm{AC}[-T,T] \to [0, \infty)$ by
 $$
 \mathscr{L}_n(u) = \int_{-T}^{T} L_n(t, u(t), u'(t))\,dt.
 $$
 Working with these approximations is much easier, since there is only a finite number of singularities in $w_n$.  So it is important to know what error we make in moving to these approximations, which is shown in the next lemma.
\begin{lemma}\label{finitejump}
Let $u \in \mathrm{AC}[-T,T]$ and  $n\geq 0$.  Then
$$
|(\scrl(u) - \scrl(w)) - (\scrl_n (u) - \scrl_n (w_n))| <\frac{ T_{n+1}^2}{2}.
$$
\end{lemma}
\begin{proof}
We first estimate  $|\scrl (u) - \scrl_n(u)|$.
Recall our definitions of $m_n>n$, $M_n \geq 0$,  and $G_n \supseteq \bigcup_{i=0}^{m_n}\{x_i\}$ from above.  Let $t \in [-T,T] \backslash G_n$. We see by~\eqref{lipphin} and~\eqref{Mn} that
$$
\mathrm{Lip}(\phi_{m_n}(t, .))\leq \sum_{i=0}^{m_n} \left(\max\{\psi_i (t), \psi_i (x_i + T_i)\}\right) \leq M_n.
$$
Then using~\eqref{w-wn} and~\ref{r3} that
$$
|\phi_{m_n}(t, u-w) - \phi_{m_n}(t, u-w_n)|  \leq M_n \|w-w_n\|_{\infty}  \leq 20 M_n R_{n+1} \leq  \frac{T_{n+1}^2}{16}.
 $$
Then by~\eqref{phi-phin} and~\eqref{mn}, for all $t \in [-T,T] \backslash G_n$ we have
\begin{align*}
|\phi(t, u - w) - \phi(t, u- w_n)| &\leq |\phi(t, u-w) - \phi_{m_n} (t, u-w)| \\&\phantom{\leq} {}+ |\phi_{m_n}(t, u-w) - \phi_{m_n}(t, u-w_n)| \\&\phantom{\leq} {}+ |\phi_{m_n}(t, u-w_n) - \phi(t , u-w_n)|\\
 &\leq 2 \|\phi - \phi_{m_n}\|_{\infty} + \frac{T_{n+1}^2}{16} \\
 &< 2 \cdot 2^{-m_n} + \frac{T_{n+1}^2}{16}\\
 &<  \frac{T_{n+1}^2}{8}.
 \end{align*} So
 $$
 \int_{[-T,T] \backslash G_n} |\phi(t, u - w) - \phi(t, u- w_n)| \leq \frac{T_{n+1}^2}{8}.
 $$
Now,  using \eqref{phibd} and \eqref{Fnmeas}, we see
$$
\int_{G_n} |\phi(t, u - w) - \phi(t, u - w_n)|  \leq  2 \int_{G_n} \|\phi\|_{\infty}
 \leq  2C\, \mathrm{meas}(G_n)
 \leq  \frac{T_{n+1}^2}{8}.
 $$
Combining, we have
\begin{equation}
|\scrl(u) - \scrl_n (u)|  \leq \int_{-T}^{T} |\phi(t, u-w) - \phi(t, u-w_n)| \leq  \frac{T_{n+1}^2}{4} \label{finitejump1}.
\end{equation}

Now we estimate $|\scrl(w) - \scrl_n (w_n)|$.  For a.e. $t \in (\bigcup_{i=n+1}^{m_n}Y_i) \backslash (\bigcup_{i=n+1}^{m_n}Z_i)$, we have by~\eqref{cvg'} and~\ref{t3} that
$$
|w_n'(t) - w_{m_n}'(t)| \leq \left( \sum_{i= n+1}^{m_n}|w_i'(t) - w_{i-1}'(t)|\right) \leq \sum_{i=n+1}^{m_n} \frac{T_i^2}{128} \leq \frac{T_{n+1}^2}{64}.
$$
By~\eqref{lipwn},~\ref{r2}, and~\ref{r3}, we have
$$
\int_{\bigcup_{i= n+1}^{m_n}Z_i}|w_n' - w_{m_n}'| \leq 4\, \mathrm{meas}\left( \bigcup_{i=n+1}^{m_n}Z_i \right) \leq 4 \left(\sum_{i=n+1}^{m_n} 2 R_i\right) \leq 16 R_{n+1} \leq \frac{T_{n+1}^2}{64}.
$$
Thus, using~\eqref{wn'tow'},
\begin{align*}
\int_{(\bigcup_{i=n+1}^{\infty}Y_i)\backslash (\bigcup_{i=m_n + 1}^{\infty}Y_i)}|w_n' - w'| & = \int_{(\bigcup_{i=n+1}^{\infty}Y_i)\backslash (\bigcup_{i=m_n + 1}^{\infty}Y_i)}|w_n' - w_{m_n}'|\\
& \leq   \int_{\bigcup_{i=n+1}^{m_n}Y_i} |w_n' - w_{m_n}'| \\
&\leq  \frac{T_{n+1}^2}{32}.
\end{align*}
On the other hand, by~\eqref{lipwn},~\eqref{lipw},~\eqref{iv}, and~\eqref{mn},
\begin{align*}
\int_{\bigcup_{i=m_n + 1}^{\infty}Y_i} |w_n' - w'|  & \leq  4\, \mathrm{meas}\left( \bigcup_{i=m_n + 1}^{\infty} Y_i \right)  \\& \leq   4 \left(\sum_{i= m_n + 1}^{\infty} 2 T_i \right)\\& <   8 \left(\sum_{i= m_n + 1}^{\infty}  2^{-i}\right)\\& =   8\cdot 2^{-m_n} \\ & <  \frac{T_{n+1}^2}{32}.
\end{align*}
Hence by~\eqref{phi0is0},~\eqref{wn'tow'},~\eqref{lipwn}, and~\eqref{lipw},
\begin{equation}
|\scrl (w) - \scrl_n (w_n)|  \leq \int_{-T}^{T} | (w')^2 - (w_n')^2| \leq 4 \int_{\bigcup_{i=n+1}^{\infty}Y_i}|w_n' - w'|
 < \frac{T_{n+1}^2}{4} \label{finitejump2}.
\end{equation}
Combining the two estimates~\eqref{finitejump1} and~\eqref{finitejump2} gives the result.
\end{proof}
 We now show that $w$ is the unique minimizer of ($\star$).   We briefly discuss the main ideas behind the proof, which as mentioned before, are essentially those of the proof that $\tw$ minimizes the variational problem with ``basic'' Lagrangian
 $$
 (t, y, p) \mapsto \tL(t,y,p) = \tp(t, y-\tw(t)) + p^2.
   $$
   So suppose for now $\tu \in \mathrm{AC}[-T,T]$ is a minimizer for this basic problem with  Lagrangian $\tL$.  If $\tu(0)=\tw(0)$, it suffices to argue separately on $[-T,0]$ and $[0, T]$.  We consider $[0,T]$.  But $\tw$ is  $C^{\infty}$ on $(0, T)$, so we can make the important step of integrating by parts.  Moreover, a simple trick relying on $\tu$ being a minimizer gives us that $|\tu(t)| \leq |t|$ (see Lemma~\ref{lemma1} below for the essence of the argument), so $|\tu(t) - \tw(t)| \leq 2 |t|$.  Note that for any two functions $\bar{u}, \bar{w} \in \mathrm{AC}[-T,T]$, we have
 \begin{equation}
 (\bar{u}')^2 - (\bar{w}')^2 = (\bar{u}' - \bar{w}')^2 + 2 (\bar{u}' - \bar{w}')\bar{w}' \geq 2 (\bar{u}' - \bar{w}')\bar{w}'. \label{triv}
 \end{equation}
 So we can  argue
\begin{align*}
\int_0^T \big(\tp(t, \tu-\tw) + (\tu')^2 \big)- \int_0^T (\tw')^2 &\geq \int_0^T \big(2(\tu'-\tw')\tw' + \tp(t, \tu -\tw)\big)\\
 & =[2(\tu-\tw)\tw']_0^T  \\ &\phantom{=} -{}\int_0^T \big(2(\tu - \tw)\tw'' + \tp(t, \tu - \tw)\big) \\
& \geq \int_0^T \big(\psi(t) |\tu - \tw| - 2 |\tu - \tw|| \tw''(t)|\big)
\end{align*}
and hence it suffices to choose $\psi$ large enough to dominate $\tw''$, which we can do (this is the role of $\psi^2$).  
This argument cannot be performed in the case when $\tu(0) \neq \tw(0)$, and there is no \textit{a priori} reason why this might not occur.  In this case, we compare $\tu$ not with $\tw$ but with a new function we obtain by replacing $\tw$ with a linear function on an interval around $0$
.

This basic idea on $\tw$ is mimicked locally on $w$ around each $x_n$; more precisely we in fact argue with $w_n$ and then either show that for some $n$ this suffices to give the result for $w$, or pass to the limit.  The techniques of our proof show in fact that $w_n$ is the unique minimizer of the variational problem
$$
\mathrm{AC}[-T,T] \ni u \mapsto \mathscr{L}_n(u)
$$
over those $u$ such that $u(\pm T) = w_n (\pm T) (= w(\pm T))$.  Thus in particular we get an example of a one-point non-differentiable minimizer: the conditions of Lemma 3.6 below always hold for $n=0$, which  already shows that Tonelli's theorem cannot hold in the continuous case.

We return to the problem proper.  Suppose now $u \in \mathrm{AC}[-T,T]$ is a minimizer for ($\star$) and $u \neq w$.  Note that a minimizer certainly exists, since $L$ is continuous, and superlinear and convex in $p$.  We now make a number of estimates, with the eventual aim of showing that
$$
\mathscr{L}(u) - \mathscr{L}(w) = \int_{-T}^{T} \big((u')^2 + \phi(t, u-w) - (w')^2\big) > 0,
$$
which contradicts the choice of $u$ as a minimizer for ($\star$).  Write $v= u-w$, and $v_n = u - w_n$.  If $u(x_n) = w(x_n)$ for all $ n\geq 0$, then as discussed above the proof is an easy application of integration by parts on the complement of the closure of the points $\{x_n\}_{n= 0}^{\infty}$.  (In the case that $\{x_n\}_{n=0}^{\infty}$ forms a dense set in $[-T, T]$, we should immediately have $u=w$ by continuity, thus concluding the proof of minimality of $w$ without using either the assumption that $u$ was a minimizer or that $u \neq w$.)   Should $w(x_n) \neq u(x_n)$ for some $ n \geq 0$, further argument is required.  The next lemma shows us that \emph{since $u$ is a minimizer}, it cannot be too badly behaved around any point $x \in [-T,T]$ where $u(x) \neq w(x)$.
\begin{lemma}\label{lemma1}
Let $x \in [-T,T]$ be such that $u(x) \neq w(x)$.
Let $J \subseteq [-T,T]$ be the connected component of the set of points $t \in [-T,T]$ such that
\bq
| u(t) - w(x)| > 3|t-x| \ \textrm{for $ t \in J$}.\label{epsilon}
\eq
Note that $J$ is an open subinterval of $[-T,T]$ since $u$ and $w$ agree at $\pm T$ and so by~\eqref{lipw}
$$
\left|u(\pm T) - w(x)\right| = \left|w(\pm T) - w(x)\right| \leq 2 \left|\pm T - x\right|.
$$
So there exist $a, b >0$ be such that $J= (x- a, x+ b)$ and
\bq
|u(x-a) - w(x)| = 3a \ \textrm{and}\ |u(x+b) - w(x)| = 3b .\label{delta}
\eq
Then
 \begin{list}{(\thecount)}{\usecounter{count}}
 \item $|u'| \leq 2$ almost everywhere on $J$; and \label{u'leq2onJ}
\item$|u(t) - w(x) | \leq 3 |t-x|$ for $t \notin J$.\label{smalloffjx}
\end{list}
\end{lemma}
 \begin{proof}
We suppose $u(x) > w(x)$.  The argument for the case $u(x) < w(x)$ is very similar.  Let $c, d > 0$ be such that $(x-c, x+d) $ is the connected component containing $x$ such that $u(t) > w(x) + 2|t - x|$ on $(x-c, x+d)$.  So  $u(x-c) =w(x)+ 2c$, and $u(x+d) = w(x) + 2d$.  We shall firstly prove that $u$ is convex on $(x-c, x+d)$.  (In the case $u(x) < w(x)$, we would have that $u$ is concave on $(x-c, x+d)$.)  Suppose not, so there exist $t_1, t_2 \in (x-c, x+d)$,  $t_1 < t_2$ say, and $\lambda \in [0,1]$ such that $$u(\lambda t_1 + (1-\lambda)t_2) > \lambda u(t_1) + (1- \lambda ) u(t_2).$$  Let $h \colon [-T,T] \to \R$ be the affine function with graph passing through $(t_1, u(t_1))$ and $(t_2, u(t_2))$, so
$$
h(t) = \frac{u(t_2) - u(t_1)}{t_2 - t_1}(t-t_1) + u(t_1) .
$$
So we have by assumption on $t_1, t_2$ that
$$
h(\lambda t_1 + (1-\lambda)t_2) = \lambda u(t_1) + (1- \lambda) u(t_2) < u(\lambda t_1 + (1-\lambda)t_2).
$$
Passing to connected components if necessary, we can assume
 that $h(t) < u(t)$ on $(t_1, t_2)$.    That $t_1, t_2 \in (x-c, x+d)$ implies
 $$
 u(t_1) >  w(x) + 2|t_1- x|\ \textrm{and}\ u(t_2) > w(x) + 2|t_2 - x|.
 $$
 Since $t \mapsto  2|t- x|$ is convex, and $t \mapsto h(t)$ is a straight line connecting $(t_1, u(t_1))$ and $(t_2, u(t_2))$, we have that for $ t \in (t_1, t_2)$ that
 $$
 h(t) >  w(x) + 2|t-x|.
 $$
 Now,
 $$
 w(t) \leq w(x) + 2|t-x|
 $$
 for all $t \in [-T,T]$ by~\eqref{lipw}, so we have $w(t) < h(t)$ on $(t_1, t_2)$.  So on $(t_1, t_2)$ we have
 $$
 |u-w| = u-w > h-w = |h-w|
 $$
 and thus, by~\eqref{phiinc},
\beq
\phi(t, u-w) \geq \phi(t, h-w).\label{vii}
\eeq

 Since $u > h$ on $(t_1, t_2)$, where $h$ is affine, but $u=h$ at the endpoints, we know $u$ is not affine on $(t_1, t_2)$, so we have strict inequality in H\"older's inequality, thus
 \begin{align}
 \int_{t_1}^{t_2} (u')^2 & = \frac{1}{t_2 - t_1} \left(\int_{t_1}^{t_2} 1^2\right)\left( \int_{t_1}^{t_2} (u')^2 \right) \nonumber\\
& >  \frac{1}{t_2 - t_1} \left(\int_{t_1}^{t_2} u'\right)^2 \nonumber\\
& =  \frac{(u(t_2) - u(t_1))^2}{t_2 - t_1} \nonumber\\
& = \int_{t_1}^{t_2} (h')^2.\label{holder}
\end{align}
Hence defining $\hat{u}\colon [-T,T] \to \R$ by
$$
\hat{u}(t) =
\begin{cases}
u(t) & t \notin (t_1, t_2) \\h(t) & t \in (t_1, t_2)
\end{cases}
$$
gives a $\hat{u} \in \mathrm{AC}[-T,T]$ satisfying our boundary conditions, and such that, using~\eqref{holder} and~\eqref{vii},
 \begin{align*}
 \scrl(\hat{u}) & =    \int_{[-T,T] \backslash (t_1, t_2)} \big((u')^2 + \phi(t, u-w)\big) + \int_{t_1}^{t_2}\big( (h')^2 + \phi(t, h-w)\big) \\
 & <  \int_{[-T,T] \backslash (t_1, t_2)} \big((u')^2 + \phi(t, u-w)\big) + \int_{t_1}^{t_2}\big( (u')^2 + \phi(t, u-w)\big) \\
  & =  \scrl(u)
 \end{align*}
 which contradicts $u$ being a minimizer.  Hence $u$ is indeed convex on $(x-c, x+d)$.  We now claim therefore that $|u'| \leq 2$ everywhere it exists on  $(x-c, x+d)$.  Suppose there exists $t_0 \in(x-c, x+d)$ such that $u'(t_0) > 2$.  Therefore by convexity $u'(t) > 2 $ almost everywhere on $(t_0, x+d)$.  We then have
 \begin{align*}
 u(x+d) & =   u(t_0) + \int_{t_0}^{x+d} u'(s)\, ds  \\
  & >   u(t_0) + \int_{t_0}^{x+d} 2\, ds  \\
& \geq   w(x) + 2|t_0 - x| + 2|(x+d) - t_0|\\
 &\geq w(x) + 2d,
 \end{align*}
 which contradicts  the choice of $d$, since $u(x+d) = w(x) + 2d$. Similarly one gets a contradiction assuming $u'(t_0) < -2$ for some $t_0 \in (x-c, x+d)$.

Statement~\eqref{smalloffjx} of the lemma is proved using the same trick we  used above to prove convexity of $u$ on $(x - c, x + d)$.  Suppose there is a $t_0 \in (x + b, T)$ such that $u(t_0) > w(x) + 3|t_0 - x|$.  Defining affine $h \colon [-T,T] \to \R$ by
$$
h(s) = 
w(x) + 3(s - x),
$$
we see that $h(t_0 ) < u(t_0)$.  The connected component $I$ of $[-T,T]$ such that \linebreak$h < u$ on $I$ satisfies $I \subseteq (x + b, T)$, since $u(x + b) = w(x) + 3b = h(x + b)$,  and by~\eqref{lipw}, $u(T) = w(T) \leq w(x) + 2 |T - x| < h(T)$.  We have
$$
u(s) > h(s) = w(x) + 3|s-x| \geq w(x) + 2|s-x| \geq w(s)
$$
for $ s \in I$, thus $|u-w| = u-w \geq h-w =|h-w|$.  Hence we can perform the same trick as before, constructing a new function $\hat{u} \in \mathrm{AC}[-T,T]$ by replacing $u$ with $h$ on $I$, such that $\mathscr{L}(\hat{u})  < \mathscr{L}(u)$, which again contradicts choice of $u$ as a minimizer. We can argue similarly if there exists a point $t_0 \in (-T, x-a)$ such that $u(t_0) > w(x) + 3|t_0 - x|$, and also if there exists a point $t_0 \in [-T,T] \backslash J$ with $u(t_0) <  w(x) - 3|t_0 - x|$.
  \end{proof}

Thus we see that if for some $x \in [-T,T]$, $u(x) \neq w(x)$, then $u$ must be Lipschitz on a neighbourhood of $x$, and its graph cannot escape the cone bounded by the graphs of $t \mapsto w(x)\pm 3|t-x|$ off this neighbourhood.  We note that the second conclusion of the Lemma holds by the same argument even in case $u(x) = w(x)$ and thus when the set $J$ introduced is empty.

For the remainder of the proof, we assume that $u(x_n) \neq w(x_n)$ for all $n \geq 0$ .  If not one can just perform the following argument on the connected components of $[-T,T] \backslash \cl{\{x_n : u(x_n) = w(x_n) \}}$.   We make remarks in the proofs of Lemma~\ref{ngood} and Corollary~\ref{cor} at those points where a note of  additional argument is required in the   general case.

For each $n \geq 0$ we now introduce some definitions and notation.  Let $a_n, b_n > \nolinebreak0$ be such that $J_n := (x_n - a_n, x_n + b_n)$ is the connected component of $[-T,T]$ containing $x_n$ such that $|u(t) - w(x_n)| > 3|t-x_n|$ for $ t \in J_n$, as in Lemma~\ref{lemma1}.  So
\bq
|u(x_n-a_n) - w(x_n)| = 3a_n, \ \textrm{and}\ |u(x_n+b_n) - w(x_n)| = 3b_n .
\eq
We let $c_n = \max\{a_n, b_n\}$, and write $\tJ_n = [x_n - c_n, x_n + c_n]$.   We note the following immediate corollary of Lemma~\ref{lemma1}.  Fix $ n \geq 0$.  For $ t \notin J_n$, we have for any $i \geq n$, by~\eqref{wfixxi},~\eqref{smalloffjx},  and~\eqref{lipwn} that
\begin{align}
|v_i(t)| & \leq |u(t) - w(x_n)| + |w(x_n) - w_i(t)|\nonumber\\
 & =  |u(t) - w (x_n)| + |w_i(x_n)- w_i (t)| \nonumber \\
 & < 5 |t - x_n|.\label{vileq5offJn}
 \end{align}

 Easy considerations of the graphs of the two Lipschitz functions give the following lower bounds of $|v_n|$ on $J_n$; the interval $J_n$ was defined precisely to ensure such constant lower bounds, i.e. that the graph of  putative minimizer $u$  cannot get too  close to that of $w$ around $x_n$. Let $i \geq n-1$, then $w_i(x_n) = w(x_n)$, so
\begin{align}
|v_i(t)| &\geq a_n\ \textrm{for $t \in [x_n - a_n, x_n]$};\ \textrm{and}\label{upsilon2a}\\
 |v_i(t)| & \geq b_n \ \textrm{ for $t \in [x_n, x_n + b_n]$}\label{upsilon2b}.
\end{align}
As we see next, this lower bound means we have a certain amount of weight concentrated in our Lagrangian around any $x_n$.  The total weight is of course in general even larger---we took an infinite sum of such non-negative terms---but the important term is the $\tp_n$ term which deals precisely with the oscillations introduced by $w_n$ to get singularity of $w$ at $x_n$.
\begin{lemma}\label{intphibig}
Let $n \geq 0$, and suppose $\tJ_n \subseteq Y_n$.  Then
$$
\int_{\tJ_n} \tp_n^1(t, v_n) \geq \frac{201 c_n}{\two 1/c_n}.
$$
\end{lemma}
\begin{proof}
Suppose $b_n \geq a_n$.  The case $a_n > b_n$ differs only in trivial notation.  So \linebreak$c_n = b_n$, and~\eqref{upsilon2b} implies that on $[x_n, x_n + b_n /5]$  we have $|v_n(t)| \geq 5|t-x_n|
$, so here $\tp_n^1 (t, u- w_n) = 5 |t- x_n| \psi_n^1(t)$ by definition.  Since $t \mapsto \frac{1}{\two 1/5|t-x_n|}$ is a concave function on $[x_n, x_n + b_n / 5]$, we can estimate the integral as follows, and see using the definition of $\psi_n^1$ that
\begin{align*}
\int_{\tJ_n} \tp_n^1 (t, v_n) & \geq  \int_{x_n}^{x_n + b_n / 5} 5|t- x_n| \psi_n^1 (t)\\
& =  \int_{x_n }^{x_n + b_n / 5} \frac{5 \cdot 402}{\two 1/5|t-x_n|} \\
& \geq  \frac{1}{2} \frac{b_n}{5}\bigg(\frac{5 \cdot 402 }{ \two 1/b_n} \bigg)\\
 &=  \frac{201  b_n}{\two 1/b_n}.
\qedhere
\end{align*}
\end{proof}
 For $n\geq 0$ we define $H_n \subseteq [-T,T]$ by
$$
H_n : = \tJ_n \cap [x_n - \tau_n , x_n + \tau_n] = [x_n - d_n , x_n + d_n],\ \textrm{say},
$$
so $d_n \leq c_n$.  Note that
 $$
 w_n(x_n \pm d_n) = \alpha_n \tw_n (x_n \pm d_n) + \beta_n;\ \textrm{and}\ \tw_n'(x_n \pm d_n) = \alpha_n \tw_n'(x_n \pm d_n).
 $$
 We cannot immediately mimic the main principle of the proof and  integrate by parts across $x_n$, since $\tw'_n$ does not exist at $x_n$.  This singularity is of course the whole point of the example.  \label{ltrick}The main trick of the proof was in making the oscillations of $\tw_n$ near $x_n$  slow enough so that we can replace this function with a straight line on an interval containing $x_n$.  We can then use parts either side of this interval, and inside the interval exploit the fact that we have now introduced a function with constant derivative.  We incur an error in the boundary terms, of course, as we in general introduce discontinuities of the derivative where the line meets $\tw_n$, but the function $\tw_n$ oscillates slow enough that this error can be dominated by the weight term in the Lagrangian (the role of $\psi_n^1$).

So let $\tl_n\colon [-T,T] \to \R$ denote the affine function with graph connecting  \linebreak$(x_n - d_n, \tw_n (x_n - d_n))$ and $(x_n + d_n,\tw_n (x_n + d_n))$, i.e.
$$
\tl_n (t) = \tl_n' (t - (x_n - d_n)) + \tw_n (x_n - d_n),
$$
where
\beq
\tl_n' = \frac{\tw_n (x_n + d_n) - \tw_n (x_n - d_n)}{2d_n} = \sin \three 1/d_n.\label{ix}
\eeq
So note by~\eqref{lipwn} that
\beq
|\alpha_n \tl_n' | \leq \mathrm{Lip}(w_n) <  2. \label{modtln'}
\eeq
Define $l_n \colon [-T,T] \to \R$ by
$$
l_n (t) =
\begin{cases}
w_n(t) & t \notin H_n \\
\alpha_n \tl_n (t) + \beta_n & t \in H_n .
\end{cases}
$$
Clearly $l_n \in \mathrm{AC}[-T,T]$.

We shall find the following notation useful, representing the boundary terms we get as a result of integrating by parts, firstly inside $H_n$, integrating $l_n' v_n'$, and secondly outside $H_n$, integrating $w_n' v_n'$:
\begin{align*}
I_{n, l} = l_n'v_n(x_n -d_n),&\ I_{n, r} = l_n' v_n(x_n + d_n);\\E_{n, l} = w_n' (x_n - d_n) v_n (x_n -d_n), &\ E_{n, r} = w_n' (x_n + d_n) v_n (x_n + d_n).
\end{align*}
Note that
\begin{align}
|I_{n,l} - E_{n,l} |&= |\alpha_n||v_n (x_n - d_n) (\tl_n' - \tw_n'(x_n - d_n))|;\label{pil}\ \textrm{and} \\|I_{n,r} - E_{n,r}| &= |\alpha_n||v_n(x_n + d_n) (\tl_n' - \tw_n'(x_n + d_n))|.\label{pir}
\end{align}

\begin{lemma}\label{u'-wn'overHn}
Let $n \geq0$.
Then
$$
\int_{H_n} (u')^2 
 - (w_n')^2 > 2(I_{n,r} - I_{n,l}) - \frac{160d_n}{\two1/d_n}.
 $$
\end{lemma}
\begin{proof}
We want to use the following estimate, replacing $w_n$ with the line $l_n$ and estimating the error:
\begin{align}
\int_{H_n} (u')^2 - (w_n')^2 & =  \int_{H_n} \big((u')^2 - (l_n')^2 \big)+ \int_{H_n}\big((l_n')^2 - (w_n')^2 \big)\nonumber\\
& \geq  \int_{H_n} \big( (u')^2 - (l_n')^2 \big) - \int_{H_n} |(l_n')^2 - (w_n')^2| .\label{linetrick}
\end{align}
Since $w_n' = \alpha_n \tw_n'$ and $l_n'= \alpha_n \tl_n'$ on $H_n$, a factor of $|\alpha_n^2|\leq 4$ comes out of the second, error term, so  we can just estimate this term in the case $n=0$; the case of general $n$ is just a translation of this base case.  We drop the index 0 from the notation.

Observe that for $t > 0$, we have
$$
\frac{d}{dt}\left(\sin \three 1/|t|\right) = -\frac{\cos \three 1/|t|}{t(\two 1/|t|)( \one 1/|t|)},
$$
so
$$
\left| \frac{d}{dt}\left(\sin \three 1/|t|\right)\right| \leq \frac{1}{t(\two 1/|t|)( \one 1/|t|)}.
$$
Hence by applying the mean value theorem we can see  for $0 < t < d$, recalling~\eqref{ix} and~\eqref{lamda}, that 
\begin{align}\lefteqn{
|\tl' - \tw'(t)|}\nonumber \\
 & = \bigg| (\sin \three 1/d)  - \left( ( \sin \three 1/|t|) - \frac{\cos \three 1/|t|}{(\two 1/|t|)(\one 1/|t|)}\right) \bigg| \nonumber\\
 & \leq |(( \sin \three 1/d) - ( \sin \three 1/|t|))| + \frac{1}{(\two 1/|t|)(\one 1/|t|)} \label{tln'-twn'}\\
& \leq \frac{(d-t)}{t(\two 1/d)( \one 1/d)} + \frac{1}{(\two 1/d)(\one1/d)}\nonumber \\
&= \frac{d}{t (\two 1/d)(\one 1/d)}.\nonumber
\end{align}
Then for $t \in (\frac{d}{\one 1/ d}, d)$,  we have
$$
|\tl' - \tw'(t)| < \frac{1}{\two 1/d };
$$
the function $\tw$ oscillates slowly enough that a good estimate for the discontinuity of the derivative holds on an interval in the domain of integration large enough in measure.  Since $\tw'$ is even, we can estimate as follows, using
~\eqref{modtln'} and~\eqref{lipwn}:
 \begin{align}
 \int_{H} |(\tl')^2 - (\tw')^2|  &= 2 \int_0^d  | \tl' - \tw ' ||\tl' + \tw'| \nonumber\\
 &\leq  8 \left( \int_0^{\frac{d}{\one 1/d}} |\tl' - \tw'|  + \int_{\frac{d}{\one 1/d}}^d |\tl' - \tw'|\right)\nonumber \\
 &<  8\left(\frac{4d}{\one1/d} + \int_{\frac{d}{\one 1/d}}^d \frac{1}{\two 1/d}\right)\nonumber \\
& \leq 8 \left( \frac{4d}{\one 1/d} + \frac{d}{\two 1/d}\right) \nonumber\\
& \leq \frac{40d}{\two 1/d}\label{ln'2-wn'2}.
\end{align}

By~\eqref{triv} we have
$$
\int_{H_n} \big((u')^2 - (l_n')^2\big)  \geq    2  l_n' [ u - l_n]_{x_n-d_n}^{x_n +d_n} =  2(I_{n,r} - I_{n,l}).
 $$  Putting this and~\eqref{ln'2-wn'2} into~\eqref{linetrick} gives the result.
\end{proof}
 An estimate established in the preceding proof also gives easily the following important result. The errors we incur in our boundary terms by introducing a jump discontinuity in the derivative of our new function $l_n$ are sufficiently small; they can be controlled by the integral over $H_n = [x_n - d_n, x_n + d_n]$ of a continuous function in $c_n \geq d_n$ taking value $0$ at $x_n$.
\begin{lemma}\label{partserror}
Let $n \geq 0$.  Then
$$
|I_{n,r} - E_{n,r}| + |I_{n, l} - E_{n, l}| < \frac{20c_n}{(\one 1/ c_n)(\two 1/ c_n)}.
$$
\end{lemma}
\begin{proof}

We just have to estimate $|v_n (x_n \pm d_n)|$.  Suppose $u(x_n) > w(x_n)$; the argument for $u(x_n)<w(x_n)$ is similar.  Suppose also $b_n \geq a_n$, so $c_n = b_n$.  The case $a_n > b_n$ is similar.  Then $u(t) \leq u(x_n + b_n)  
$ by convexity of $u$, for all $t \in J_n$.  If $x_n - d_n \notin J_n$, then~\eqref{vileq5offJn} gives us the immediate estimate  $|v_n (x_n - d_n)| \leq 5d_n\leq 5 b_n$ since $d_n \leq b_n$.  If $x_n - d_n \in J_n$, then we can argue that, since certainly $x_n + d_n \in J_n$,
$$
w(x_n) < w(x_n) + 3 d_n \leq u (x_n \pm d_n) \leq  u(x_n + b_n) = w(x_n) + 3b_n
$$
thus
$$
0 < u(x_n \pm d_n) - w(x_n) \leq 3 b_n.
$$
Hence using~\eqref{wfixxi} and~\eqref{lipwn}, and since $d_n \leq b_n$,
\begin{align*}
 |v_n (x_n \pm d_n)| &\leq |u(x_n \pm d_n) - w(x_n)| + |w_n(x_n) - w_n(x_n \pm d_n)|\\
& \leq 3 b_n + 2d_n \\
&\leq 5 b_n.
\end{align*}

The result then follows  using the estimate~\eqref{tln'-twn'} for $t=d$ in~\eqref{pir} and~\eqref{pil}, and since $|\alpha_n | < 2$ and $d_n \leq c_n$.
\end{proof}
We now combine our estimates for $\mathscr{L}_n$ across the whole domain $[-T,T]$, integrating by parts off $\bigcup_{i=1}^n H_i$ and using the above estimate on each $H_i$.  We work with simplifying assumptions implying the relevant intervals do not overlap.  We discuss later how to deal with the failure of these assumptions.
\begin{lemma}\label{ngood}
Suppose  $n \geq 0$ is such that for all $0 \leq j \leq n$,
\begin{gather}
\tJ_k \cap Y_j = \emptyset\ \textrm{ for all $ 0 \leq k < j$; and} \label{mu}\\  \tJ_j \subseteq Y_j. \label{xi}
\end{gather}
Then
$$
  \scrl_n (u) - \scrl_n (w_n) \geq  \sum_{i=0}^n \left(\frac{c_i}{\two 1/ c_i}\right)  + \int_{[-T,T] \backslash \bigcup_{i=0}^n H_i} |v_n|.
  $$
\end{lemma}
\begin{proof}
By~\eqref{wn=wn-1} and assumption~\eqref{mu} we have $w_j = w_k$ on $\tJ_k$ for all
$ 0 \leq k < j \leq n$, in particular
\beq
 w_n = w_k,\ w_n' = w_k'\ \textrm{and}\ w_n'' = w_k''\ \textrm{(wherever both sides exist)}\ \textrm{on $\tJ_k$}.\label{wn=wk}
 \eeq  Also, by assumptions~\eqref{xi} and~\eqref{mu} together we have that for $0 \leq k < j \leq n$
$$
\tJ_k \cap \tJ_j \subseteq \tJ_k \cap Y_j = \emptyset,
$$
 i.e. the   $\{\tJ_i\}_{i=0}^n $
 are pairwise disjoint.

Now, let $0 \leq i \leq n$.    We see, using~\eqref{triv}, that
\begin{align*}
\lefteqn{
\int_{\tJ_i} \big((u')^2 + \phi(t, v_i) - (w_i')^2\big)
}\\
& =  \int_{\tJ_i} \phi(t, v_i) + \int_{\tJ_i \backslash H_i} \big((u')^2  - (w_i')^2\big)  + \int_{H_i} \big((u')^2  - (w_i')^2\big)  \\
& \geq  \int_{\tJ_i} (\phi^1 (t, v_i) + \phi^2 (t, v_i) ) + \int_{\tJ_i \backslash H_i} 2v_i' w_i' +  \int_{H_i} \big((u')^2  - (w_i')^2\big)\\
& \geq  \int_{\tJ_i \backslash H_i} (\phi^2(t, v_i) + 2v_i' w_i') + \int_{\tJ_i} \phi^1(t, v_i) + \int_{H_i} \big((u')^2  - (w_i')^2\big).\end{align*}
Now, by Lemma~\ref{intphibig} (note this applies by assumption~\eqref{xi}) and Lemma~\ref{u'-wn'overHn}, and since $c_i \geq d_i$, \begin{align*} \int_{\tJ_i} \phi^1 (t, v_i) + \int_{H_i} ((u')^2 - (w_i')^2) & \geq  \int_{\tJ_i} \tp_i^1(t, v_i) + \int_{H_i} ((u')^2 - (w_i')^2) \\
& \geq  \frac{41c_i}{\two 1/c_i} + 2(I_{i,r} - I_{i,l}).
\end{align*}
So combining we have
\begin{equation}
 \int_{\tJ_i} (u')^2 + \phi(t, v_i) - (w_i')^2 \geq   \frac{41c_i}{\two 1/c_i} + 2(I_{i,r} - I_{i,l}) + \int_{\tJ_i \backslash H_i} (\phi^2(t, v_i) + 2v_i' w_i').\label{Jicomplete2}
 \end{equation}

Now, for any $t \in [-T,T]$, write $\mathcal{I}_n (t) = \{ i = 0, \ldots, n : t \in Y_i\}$.
We show by an easy induction that for almost every $ t \in [-T,T]$,
\beq
\sum_{i \in \mathcal{I}_n(t)}\psi_i^2(t) \geq 2|w_n''(t)| + 1 + 2^{-(n-1)}\label{omicron}.
\eeq
For $n = 0$, we have by definition that for all $ t \neq x_0$,
$$
\psi_0^2 (t) =  3 + 4|w_0''(t)| \geq 3 + 2|w_0''(t)|
$$
as required.  Suppose the result holds for all $0 \leq i \leq n-1$, where $ n \geq 1$.  Let $i = i(n,t) \leq n$ denote the greatest index in $\mathcal{I}_n (t)$, i.e. the greatest index $i$ such that $t \in Y_i$.
By~\eqref{wn=wn-1} we have $w_n''(t) = w_{i}''(t)$ whenever both sides exist, i.e. almost everywhere.  If $t \in (x_{i} - \tau_{i}, x_{i} + \tau_{i})$, then $w_{i}''(t) =  \alpha_i\tw_{i} ''(t)$ by~\eqref{wn=twn}, and by definition, for $t \neq x_{i}$,
$$
\sum_{j \in \mathcal{I}_n (t)} \psi_j^2(t) \geq \psi_{i}^2 (t) = 3 + 4|\tw_{i}''(t)| \geq 1 + 2^{-(n-1)} + 2|\alpha_i \tw_i''(t)|
$$
 as required.  If $t \notin [x_{i} - \tau_{i}, x_{i} + \tau_{i}]$ (note then necessarily $i \geq 1$ since $\tau_0 = T_0 = \nolinebreak T$), then $|w_{i} ''(t)| \leq |w_{i-1}''(t)| + 2^{-i}$ almost everywhere by \eqref{wn''-wn-1''} so by inductive hypothesis
 \begin{align*}
  \sum_{j \in \mathcal{I}_n(t)} \psi_j^2 (t) & \geq  \sum_{j \in \mathcal{I}_{i-1}(t)} \psi_j^2 (t) \\
& \geq  2 |w_{i-1}''(t)| + 1 + 2^{-((i-1)-1)} \\
& \geq  2|w_{i}''(t)| - 2 \cdot 2^{-i} + 1 + 2^{-((i-1)-1)} \\
& \geq  2|w_n''(t)| + 1 + 2^{-(n-1)}
\end{align*} as required for~\eqref{omicron}.

Given this, now consider $ t \notin \bigcup_{i=0}^n \tJ_i$.  Then since $\tJ_i \supseteq J_i$ for all $i \geq 0$,~\eqref{vileq5offJn} gives that $|v_n (t)| \leq 5|t - x_i|$ for all $0 \leq i \leq n$.  Therefore $\tp_i^2 (t, v_n) = |v_n| \psi_i^2 (t)$ by definition for $ i \in \mathcal{I}_n (t)$.  Thus almost everywhere, we have by \eqref{omicron} that
\begin{align*}
 \phi^2(t, v_n) - 2v_n w_n'' & \geq  \sum_{i \in \mathcal{I}_n (t)} (\tp_i^2 (t, v_n)) - 2|v_n| |w_n''| \\
& =  \sum_{i \in \mathcal{I}_n (t)} (\psi_i^2 (t)|v_n|) - 2|v_n| |w_n''| \\
& =  |v_n| \left(\sum_{i \in \mathcal{I}_n (t)} (\psi_i^2 (t)) - 2|w_n''(t)|\right) \\
& >  |v_n|.
\end{align*}

Now, let $ t \in \tJ_i \backslash H_i$.  Again note that we must have $i \geq 1$, since $\tau_0 = T_0 = \nolinebreak T$.  Since $\{\tJ_j\}_{j=0}^n $ are pairwise disjoint, we have that $ t \notin \tJ_j$ for $j < i$.  Hence, again by~\eqref{vileq5offJn}, $|v_i| \leq 5|t-x_j|$ for all $ j < i$, so by definition $\tp_j^2(t, v_i) = \psi_j^2 (t) |v_i|$ for $j \in \mathcal{I}_{i-1}(t)$.   Since $t \notin H_i$, we have $t \notin [x_i - \tau_i, x_i +,\tau_i]$, and hence that $|w_i''(t)| \leq |w_{i-1}''(t)| + 2^{-i}$ almost everywhere by~\eqref{wn''-wn-1''}.  Hence by \eqref{omicron} we have almost everywhere
\begin{align*}
\smash{\sum_{j \in \mathcal{I}_{i-1} (t)}\psi_j^2(t)} & \geq  1 + 2|w_{i-1}''(t)| + 2^{-(i-2)} \\
& \geq  1 + 2|w_i''(t)| - 2^{-(i-1)} + 2^{-(i-2)} \\
& >  1 + 2|w_i ''(t)|,
\end{align*}
and so
\begin{align*}
\phi^2(t, v_i) - 2v_iw_i'' &\geq \sum_{j \in \mathcal{I}_{i-1} (t)}(\tp_j^2 (t, v_i)) - 2|v_i ||w_i''| \\
&= \sum_{j \in \mathcal{I}_{i-1} (t)}(\psi_j^2 (t)|v_i|) - 2|v_i||w_i''|\\
& > |v_i|.
\end{align*}
Thus we have for almost every $t \notin \bigcup_{i=0}^n H_i$, noting the argument on $\tJ_i \backslash H_i$ above applies by~\eqref{wn=wk},  that
$$
\phi^2 (t, v_n) - 2v_n w_n'' > |v_n|,
$$
 and hence
 \beq
 \int_{[-T,T] \backslash \bigcup_{i=0}^n H_i} \left(\phi^2 (t, v_n) - 2v_n w_n''\right) \geq \int_{[-T,T] \backslash \bigcup_{i=0}^n H_i} |v_n|.\label{dblstar}
 \eeq

The reason for making this estimate is that we want to integrate $v_n' w_n'$ by parts on $[-T,T] \backslash \bigcup_{i=0}^n H_i$.  Under our standing assumption that $u(x_i) \neq w(x_i)$ for all $i \geq 0$, we see immediately that this is possible, since $v_n$ and $w_n'$ are bounded and absolutely continuous on $[-T,T] \backslash \bigcup_{i=0}^n H_i$ by~\eqref{wn'lip}, and thus $v_n w_n'$ is absolutely continuous on $[-T,T] \backslash \bigcup_{i=0}^n H_i$.  However,  in the general case that $w(x_j) = u(x_j)$ for some $0 \leq j \leq n$, and thus that $w_n (x_j) = u(x_j)$,  we have to argue a little more.

We claim that even in this general case the parts formula is still valid on $[-T,T] \backslash \bigcup_{i=0}^n H_i$, this is the assertion that $v_n w_n'$ can be written as an indefinite integral on  $[-T,T] \backslash \bigcup_{i=0}^n H_i$.   The argument of the preceding paragraph gives us that $v_n w_n'$ is absolutely continuous on subintervals bounded away from all $x_j$ with $u(x_j) = w(x_j)$.  Thus for each $0 \leq j \leq n$ such that $u(x_j) = w(x_j)$, and hence $H_j = \emptyset$, it suffices to check that $v_n w_n'$ can be written as an indefinite integral on a neighbourhood $U=(x_j - \delta, x_j + \delta) \subseteq [x_j - \tau_j, x_j + \tau_j]$ of $x_j$ not containing any other points $x_i$ for $0 \leq i \leq n$.  We check that
$$
\int_{x_j - \delta}^{x_j} (v_n w_n')'(s) \, ds =  - (v_n w_n')(x_j - \delta),
$$
the corresponding equality on the right of $x_j$ follows similarly.   We know that $v_n w_n'$ is absolutely continuous on subintervals of $U$ bounded away from $x_j$.  We claim that  $(v_n w_n')' \in L^1(U)$.  Given this, we can use the DCT to get the required result: since $v_n$ is continuous and $v_n (x_j) = 0$, we use~\eqref{lipwn} to see that
\begin{align*}
 - (v_nw_n')(x_j - \delta)& = \lim_{t \to x_j} ((v_n w_n')(t) - (v_n w_n')(x_j - \delta)) \\
&=\lim_{t \to x_j} \int_{x_j - \delta}^t (v_n w_n')'(s) \, ds\\
& = \int_{x_j - \delta}^{x_j}(v_n w_n')'(s) \, ds.
\end{align*}
To see $(v_n w_n')' \in L^1(U)$, note that since $u$ is by choice a minimizer for ($\star$), we have by~\eqref{lipw}
  $$
  \int_{-T}^{T} (u')^2 \leq \mathscr{L}(u) \leq \mathscr{L}(w) = \int_{-T}^{T} (w')^2 < \infty.
  $$
   Also, we can prove that $|u| \leq 3 |t - x_j|$ everywhere on $[-T,T]$, for example by noting the arguments used to prove~\eqref{smalloffjx} still apply when $J_j = \emptyset$.  So using~\eqref{wn=twn} and~\eqref{lipwn}, we have
 \begin{align*}
  \int_U |(v_n w_n')'| & \leq \int_U |v_n w_n''| + \int_U |v_n' w_n'| \\
 & \leq \int_U |uw_n''| + \int_U |w_n w_n''| + 2\left(\int_U |u'| + 2\right) \\
 & \leq  |\alpha_j|\left(3\int_U |(t-x_j) \tw_j''| + |\alpha_j|\int_U |\tw_j \tw_j''|\right) + 2\left(\int_U |u'| + 2\right)\\
 & \leq 2 \left(3 \sup_{t \in U}|(t - x_j)\tw_j''(t)|+ 2\sup_{t \in U}|(t - x_j)\tw_j''(t)|+ \int_U |u'| + 2\right)
\end{align*}
This right hand side is finite by~\eqref{wn''cont},~\eqref{alpha}, and the above note.

So, using~\eqref{triv}, and recalling that $v_n (\pm T) = 0$,
and using~\eqref{dblstar}, we have, integrating by parts as we now know we can do,  that
\begin{align}
\lefteqn{
 \int_{[-T,T] \backslash \bigcup_{i=0}^n H_i} \big(\phi^2 (t, v_n) + (u')^2 - (w_n')^2\big)
 }\nonumber\\
  & \geq  \int_{[-T,T] \backslash \bigcup_{i=0}^n H_i}\left( \phi^2 (t, v_n) + 2v_n' w_n'\right) \nonumber\\
& =  2[v_n w_n']_{[-T,T] \backslash \bigcup_{i=0}^n H_i} + \int_{[-T,T] \backslash \bigcup_{i=0}^n H_i}\left( \phi^2(t, v_n) - 2 v_n w_n''\right) \nonumber\\
& = -2 \sum_{i=0}^n [v_i w_i']_{x_i - d_i}^{x_i + d_i} + \int_{[-T,T] \backslash \bigcup_{i=0}^n H_i}\left( \phi^2 (t, v_n) - 2v_n w_n''\right) \nonumber\\
& \geq  -2 \sum_{i=0}^n (E_{i, r} - E_{i, l}) + \int_{[-T,T] \backslash \bigcup_{i=0}^n H_i} |v_n|. \label{offHi}
\end{align}

So since $\{\tJ_i\}_{i=0}^n$ are pairwise disjoint, we see,  using~\eqref{phi0is0},~\eqref{wn=wk}, ~\eqref{Jicomplete2},~\eqref{offHi}, and Lemma~\ref{partserror}, that
\begin{align*}
 \lefteqn{
 \scrl_n (u) - \scrl_n (w_n)
 }
\\
 & =  \sum_{i=0}^n \int_{\tJ_i} \big((u')^2 + \phi(t, v_i) - (w_i')^2\big) + \int_{[-T,T] \backslash \bigcup_{i=0}^n \tJ_i}\big( (u')^2 + \phi(t, v_n) - (w_n')^2\big)\\
& \geq  \sum_{i=0}^n\left(\frac{41c_i}{\two 1/c_i} + 2(I_{i,r} - I_{i,l}) +  \int_{\tJ_i \backslash H_i} \big( (u')^2 + \phi^2(t, v_i) - (w_i')^2\big)  \right) \\&\phantom{\geq} {}+ \int_{[-T,T] \backslash \bigcup_{i=0}^n \tJ_i} \big( (u')^2 + \phi^2(t, v_n) - (w_n')^2\big) \\
& \geq  \sum_{i=0}^n \left(\frac{41c_i}{\two 1/c_i} + 2(I_{i,r} - I_{i,l})\right) + \int_{[-T,T] \backslash \bigcup_{i=0}^n H_i} \big( (u')^2 + \phi^2(t, v_n) - (w_n')^2\big)  \\
& =  \sum_{i=0}^n \left( \frac{41c_i}{\two 1/c_i} + 2 ((I_{i, r} - E_{i, r}) - (I_{i, l} - E_{i, l})) \right) + \int_{[-T,T] \backslash \bigcup_{i=0}^n H_i} |v_n| \\
& \geq  \sum_{i=0}^n \left( \frac{41c_i}{\two 1/c_i} - 2(|I_{i, r}- E_{i, r}| + |I_{i, l} - E_{i, l}|) \right) + \int_{[-T,T] \backslash \bigcup_{i=0}^n H_i} |v_n| \\
& =  \sum_{i=0}^n \left( \frac{c_i}{\two 1/c_i}\right) + \int_{[-T,T] \backslash \bigcup_{i=0}^n H_i} |v_n|.
\qedhere
 \end{align*}
\end{proof}

\begin{corollary}\label{cor}
Suppose for all $n \geq 0$ our assumptions \eqref{mu} and \eqref{xi} hold.
 Then
 $$
 \scrl(u) - \mathscr{L} (w)  \geq \sum_{i=0}^{\infty}\left( \frac{c_i}{\two 1/ c_i}\right) + \int_{[-T,T] \backslash \bigcup_{i=0}^{\infty}H_i}  |v| > 0.
 $$
\end{corollary}
\begin{proof}  This follows from the preceding Lemma by the dominated convergence theorem, since $\mathscr{L}_n (u) - \mathscr{L}_n (w_n) \to \mathscr{L}(u) - \mathscr{L}(w)$ by Lemma~\ref{finitejump}.

We note that in the general case we do indeed have strict inequality, as is necessary for the contradiction proof.  If $u(x_n) \neq w(x_n)$ for some $n\geq 1$, then $c_n > 0$ and so the infinite sum is strictly positive.  If $u(x_n) = w(x_n)$ for all $ n\geq 1$, then $[-T,T] \backslash \bigcup_{i=0}^{\infty}H_i = [-T,T]$, so on the assumption that $u \neq w$, where both are continuous functions, the integral term must be strictly positive.
\end{proof}

The arguments of the previous lemma and its corollary relied on the intervals we have to give special attention, the $\tJ_j$, being small enough that they did not escape $Y_j$, or overlap with later $Y_k$ and hence possible $\tJ_k$.  The trick is now that should one of these assumptions fail, thus apparently making the proof more complicated,  in fact this means that we can ignore the modifications we made at stage $j$ and beyond.  That one of our assumptions fails for $j$ means that $\tJ_j$ is too large, which by the very definition of $\tJ_j$ implies the graph of $u$ is far away from that of $w$ on a set of large measure around $x_j$.  We have chosen our constants so that this large difference between $u$ and $w$ around $x_j$ gives enough weight to our Lagrangian that we can discard all modifications we made to $w_{j-1}$ and hence to $L_{j-1}$ and work just with these instead; the error so incurred is small enough that it is absorbed into this extra weight.  Very roughly, if $u$ misses $w$ at $x_j$ by an inconveniently large amount, then we don't have to worry about the fine detail of our variational problem at and beyond the scale $j$.

\begin{lemma}\label{lemma8}
Let $n\geq 1$ be such that assumptions~\eqref{mu} and~\eqref{xi} hold for $n-1$, but for some $0 \leq j < n$ we have $\tJ_j \cap Y_n \neq \emptyset$, i.e.~\eqref{mu} fails for $n$.
Then
$$
\scrl_{n-1}(u) - \scrl_{n-1}(w_{n-1}) \geq T_n^2.
$$
\end{lemma}
\begin{proof}
That~\eqref{mu} fails for $n$ implies that $c_j \geq T_n$, otherwise choosing $t \in \tJ_j \cap Y_n$ we would have by~\ref{t2} that
$$
 |x_n - x_j| \leq |x_n - t| + | t- x_j| \leq T_n + c_j < 2T_n < |x_n - x_j|.
$$
So, applying Lemma~\ref{ngood} to $n-1$ we see, using this fact, and~\eqref{tau}, that
\begin{equation*}
\scrl_{n-1}(u) - \scrl_{n-1}(w_{n-1})
 \geq  \frac{c_j}{\two 1/ c_j}
 \geq  c_j^2
 \geq  T_n^2.
\qedhere
 \end{equation*}
\end{proof}
\begin{lemma}\label{lemma9}
Let $ n\geq 1$ be such that assumption \eqref{mu} holds for $n$, assumption \eqref{xi} holds for $n-1$, but $\tJ_n \nsubseteq Y_n$, i.e.~\eqref{xi} fails for $n$.  Then
$$
\scrl_{n-1}(u) - \scrl_{n-1}(w_{n-1}) \geq T_n^2.
$$
\end{lemma}
\begin{proof}
We suppose $b_n \geq a_n$, so $c_n = b_n$.  The case $a_n > b_n$ differs only in trivial notation.  That~\eqref{xi} fails for $n$ implies that $b_n \geq T_n$.  That \eqref{mu} holds for $n$ implies
in particular that $Y_n \cap \bigcup_{i=0}^{n-1}H_i \subseteq Y_n \cap \bigcup_{i=0}^{n-1}\tJ_i = \emptyset$.  Thus by Lemma~\ref{ngood} for $n-1$,
$$
\scrl_{n-1}(u) - \scrl_{n-1}(w_{n-1})  \geq  \int_{[-T,T] \backslash \bigcup_{i=0}^{n-1}H_i}|v_{n-1}|  \geq  \int_{Y_n}|v_{n-1}| \geq  \int_{x_n }^{x_n + T_n}|v_{n-1}|.
 $$
  But the point is that $[x_n, x_n + T_n] \subseteq [x_n, x_n + b_n]$, so from~\eqref{upsilon2b} we have $|v_{n-1}| \geq b_n$ on $[x_n, x_n + T_n]$.  So we see
\begin{equation*}
 \scrl_{n-1}(u) - \scrl_{n-1}(w_{n-1})  \geq   \int_{x_n}^{x_n + T_n} b_n = T_n b_n \geq T_n^2.
\qedhere
 \end{equation*}
\end{proof}
 We can now conclude our proof that $w$ is the unique minimizer of ($\star$).  Choose the least $n \geq 0$ such that one of our crucial assumptions~\eqref{mu} or~\eqref{xi} fails.  We observe that then $ n\geq 1$ necessarily, since certainly $\tJ_0 \subseteq [-T,T]$.  If no such $n$ exists, we are in the situation of Corollary~\ref{cor} and we are done.

Suppose $n\geq 1$ is such that~\eqref{mu} fails for $n$.  Then we are in the situation of Lemma~\ref{lemma8} and we see by Lemma~\ref{finitejump} that
$$
 \scrl(u) - \scrl(w)  > \scrl_{n-1}(u) - \scrl_{n-1}(w_{n-1}) - \frac{T_n^2}{2}  \geq  \frac{T_n^2}{2}  >  0.
 $$

Suppose $n \geq 0$ is such that~\eqref{mu} holds for $n$ but~\eqref{xi} fails.  Then we are in the situation of Lemma~\ref{lemma9} and we see again by Lemma~\ref{finitejump} that
\begin{equation*}
 \scrl (u) - \scrl(w) >  \scrl_{n-1}(u) - \scrl_{n-1}(w_{n-1}) - \frac{T_n^2}{2}
 \geq  \frac{T_n^2}{2} >  0.
 \end{equation*}

\section{Singularity}
The extra oscillations we added in to $w_n$ are small enough in magnitude and far enough from $x_n$ to preserve the behaviour of $w$ as being like that of $w_n$ and hence $\tw_n$ around $x_n$.  In particular, the non-differentiability still holds.
\begin{proposition}\label{sing}
Let $n \geq 0$.  Then $\overline{D}w(x_n) \geq 1$ and $\underline{D}w(x_n) \leq -1$.
\end{proposition}
\begin{proof}
Let $ t \in [-T,T]$, and let $m > n$.  Note that if $ t \in Y_i$ for $i > n$, we have by \ref{t2}
$$
|x_n - x_i| \leq |x_n - t| + |t-x_i| \leq |x_n - t| + T_i < |x_n - t| + |x_n - x_i|/2
$$
and hence, again by condition~\ref{t2} \beq T_i < |x_n - x_i|/2 < |x_n - t|.\label{x}\eeq Now let $ t \in [-T,T]$ be such that $|t-x_n| < T_m$.
Then for $n < i \leq m$, again by~\ref{t2} and since the $T_i$ are decreasing,
$$
|t- x_i|  \geq  |x_i - x_n| - |t - x_n|  >  2T_i - T_m  \geq 2T_i - T_i = T_i,
$$
so $t \notin Y_i$ for all $n < i \leq m$.

If $t \notin Y_i$ for any $i > n$ then $w(t) = w_n (t)$ by~\eqref{wn=wn-1}, and the following argument is trivial.  Otherwise choose least $i > n$ such that $ t \in Y_i$, so $w_n (t) = w_{i-1}(t)$.  Then by the above argument we must have $i > m$, and so by~\eqref{w-wn},~\ref{r3}, and~\eqref{x},
$$
|w(t) - w_n (t)| =  |w(t) - w_{i-1}(t)|\leq\|w-w_{i-1}\|_{\infty}\leq20R_i  < 2^{-i}\,T_i < 2^{-i} |t- x_n|.
$$
Hence  we have by \eqref{wfixxi}, and since $i>m$,
$$
\left|\frac{w(t) - w(x_n)}{t-x_n} - \frac{w_n (t) - w_n (x_n)}{t-x_n} \right|  =  \left| \frac{w(t) - w_n (t)}{t-x_n} \right| \leq 2^{-i}  < 2^{-m}.
$$
Hence by~\eqref{wn=twn} and definition of $\tw_n$,
\begin{align*}
\overline{D}w(x_n) = \overline{D}w_n(x_n) = \overline{D}\alpha_n \tw_n(x_n)& \geq 1; \textrm{and}\\ \underline{D}w(x_n) = \underline{D}w_n(x_n) = \underline{D}\alpha_n \tw_n(x_n) &\leq -1.
\qedhere
\end{align*}
\end{proof}

\section{Conclusion}
The precise statement of Theorem~\ref{main} can now be obtained by letting our sequence $\{x_n\}_{n=0}^{\infty}$ be an enumeration of the rationals in $(-T,T)$.  Define
$$
\Sigma = \{ x \in (-T,T) : \overline{D}w(x) \geq 1 \ \textrm{and}\ \underline{D}w(x) \leq -1\}.
$$
Then density of $\Sigma$ is immediate by Proposition~\ref{sing}.  That it is $G_{\delta}$ is standard: $\Sigma= \bigcap_{k=1}^{\infty} (\Sigma_k^{+} \cap \Sigma_k^{-})$ where
\begin{multline*}
\Sigma_{k}^{\pm} = \Bigg\{ t \in (-T,T) : \left|\frac{w(s) - w(t)}{s-t} - \pm1\right| < 1/k \\
 \textrm{for some $s \in [-T,T]$ such that $|t-s| < 1/k$}\Bigg\}
\end{multline*}
are open sets.  That $\Sigma$ is therefore second category follows by density and Baire's theorem.\qed
\section{Further results}

It is possible to perform exactly the same type of construction to produce a continuous Lagrangian with a minimizer $w$ of the associated variational problem which has $\overline{D}w(x_n) = +\infty$ and $\underline{D}w(x_n) = -\infty$ on a given countable set $\{x_n\}_{n= 0}^{\infty}$.  The minimizer is evidently no longer Lipschitz, and so the proofs are a little harder in  technicalities, but they are similar in spirit.  The function $\tw$ on which the construction is based is in this case $\tw(t) = t (\two 1/|t|) \sin \three 1/|t|$.

In preparation is a paper performing the construction in greater generality, with $\tw(t) = t f(t) \sin h(t)$, for appropriate $f, h$.

The example presented in the present paper illustrates the main ideas, without the extra technical complications of the stronger or more general cases.

\begin{acknowledgements}
This work is part of the PhD research done by the first named author under the supervision of the second.
\end{acknowledgements}

\bibliographystyle{plain}
\bibliography{thesisbib1}   

\begin{thebibliography}{10}

\bibitem{BM}
J.~M. Ball and V.~J. Mizel.
\newblock One-dimensional variational problems whose minimizers do not satisfy
  the {E}uler-{L}agrange equation.
\newblock {\em Arch. Ration. Mech. Anal.}, 90(4):325--388, 1985.

\bibitem{BGH}
G.~Buttazzo, M.~Giaquinta, and S.~Hildebrandt.
\newblock {\em One-dimensional variational problems. An introduction}.
\newblock Oxford University Press, Oxford.

\bibitem{CV85}
F.~H. Clarke and R.~B. Vinter.
\newblock Regularity properties of solutions to the basic problem in the
  calculus of variations.
\newblock {\em Trans. Amer. Math. Soc.}, 289(1), 1985.

\bibitem{uss}
M.~Cs{\"o}rnyei, T.C. O'Neil, B.~Kirchheim, D.~Preiss, and S.~Winter.
\newblock Universal singular sets in the calculus of variations.
\newblock {\em Arch. Ration. Mech. Anal.}, 190:371--424, 2008.

\bibitem{Davie}
A.~M. Davie.
\newblock Singular minimisers in the calculus of variations in one dimension.
\newblock {\em Arch. Ration. Mech. Anal.}, 101(2):161--177, 1988.

\bibitem{Lav}
M.~Lavrentiev.
\newblock Sur quelques probl{\`e}mes du calcul des variations.
\newblock {\em Ann. Mat. Pura Appl.}, 4:7--28, 1926.

\bibitem{Mania}
B.~Mani{\`a}.
\newblock Sopra un essempio di {L}avrentieff.
\newblock {\em Bull. Un. Mat Ital.}, 13:147--153, 1934.

\bibitem{Sych91}
M.~A. Sych{\"e}v.
\newblock Regularity of solutions of some variational problems.
\newblock {\em Soviet Math. Dokl.}, 43(1):292--296, 1991.

\bibitem{Sych92}
M.~A. Sych{\"e}v.
\newblock A classical problem of the calculus of variations.
\newblock {\em Soviet Math. Dokl.}, 44(1):116--120, 1992.

\bibitem{Ton}
L.~Tonelli.
\newblock Sur un m\'ethode directe du calcul des variations.
\newblock {\em Rend. Circ. Mat. Palermo}, 39:233--264, 1915.

\end{thebibliography}

%
%

\end{document}